\documentclass[10pt]{article}

\usepackage{cite}
\usepackage{amsmath,amsthm,amsfonts,amssymb,bm}

\usepackage[light]{antpolt}    \usepackage[T1]{fontenc} 
\usepackage{color,graphicx,graphics}
\usepackage{mathrsfs,amssymb,bm,enumerate}

\newcommand{\sett}[1]{\left\{   #1   \right\}}

 \usepackage[pagebackref=true]{hyperref}

\usepackage{amsbsy}
\DeclareMathAlphabet{\mathpzc}{OT1}{pzc}{m}{it}
\newcommand{\vertiii}[1]{{\left\vert\kern-0.25ex\left\vert\kern-0.25ex\left\vert #1 
		\right\vert\kern-0.25ex\right\vert\kern-0.25ex\right\vert}}
\usepackage{epsfig}

\numberwithin{equation}{section}
\newtheorem{theorem}{\qquad Theorem}[section]
\newtheorem{lemma}[theorem]{\qquad Lemma}
\newtheorem{corollary}[theorem]{\qquad Corollary}
\newtheorem{remark}[theorem]{Remark}

\newtheorem{proposition}{Proposition}[section]
\newtheorem{definition}{Definition}[section]
\newcommand{\dd}{\;{\rm d}}
\newcommand{\ff}{\varphi}

\newcommand{\ii}{{\rm i}}

\newcommand{\nc}{\mathscr{C}}

\newcommand{\ngg}{\mathscr{G}}

\newcommand{\njj}{\mathscr{J}}
\newcommand{\nk}{\mathscr{K}}
\newcommand{\nl}{\mathscr{L}_{\vec c}}
\newcommand{\nlz}{\mathscr{L}_0}

\newcommand{\scal}[1]{\left\langle #1 \right\rangle}

\newcommand{\ns}{\Psi}

\newcommand{\fb}{\mathfrak{B}}

\newcommand{\fd}{\mathfrak{D}}

\newcommand{\fk}{\mathfrak{K}}

\newcommand{\fu}{\mathfrak{U}}
\newcommand{\la}{\langle}
\newcommand{\ra}{\rangle}

\def\d{\mbox{\rm d}}

\def\N{\mathbb{N}}

\def\F{\mathbb{F}}
\def\PP{\mathbb{P}}
\def\E{\mathbb{E}}

\def\K{\mathbb{K}}
\def\L{\mathbb{L}}

\allowdisplaybreaks

\newcommand{\lam}{\lambda}
\newcommand{\x}{\mathscr{X} }
\newcommand{\de}{\Delta}

\newcommand{\al}{\alpha}
\newcommand{\rr}{\mathbb{R}}
\newcommand{\lt}{{L^2({\mathbb{R}^n})}}

\newcommand{\ho}{{H^1(\mathbb{R}^n)}}
\newcommand{\rn}{{\mathbb{R}^n}}
\newcommand{\ltt}{\mathbb{L}_{\vec c}^{\frac12}}

\newcommand{\svc}{S_{\vec c}}
\newcommand{\vc}{{\vec c}} 
\newcommand{\vs}{{\vec\varsigma}}

\newcommand{\norm}[1]{\left\|   #1   \right\|}
\newcommand{\abso}[1]{\left|   #1   \right|}
\newcommand{\paar}[1]{\left(   #1   \right)}

\author{{\bf Amin Esfahani}\footnote{ Department of Mathematics, Nazarbayev University, Astana 010000, Kazakhstan\newline   E-mail: saesfahani@gmail.com, amin.esfahani@nu.edu.kz.}}

\title{ Angular traveling waves of the high-dimensional Boussinesq equation 
	}

\date{}

\begin{document}
	\maketitle
	\begin{abstract}
	  This paper studies traveling waves with nonzero wave speed (angular traveling waves) of the high-dimensional Boussinesq equation that have not been studied before. We analyze the properties of these waves and demonstrate that, unlike the unique stationary solution, they lack positivity, radial symmetry, and exponential decay. By employing variational and geometric approaches, along with perturbation theory, we establish the orbital (in)stability and strong instability of these traveling waves. 
 
	\end{abstract}

	
	
	
	
	\section{Introduction}
In this paper, we study traveling wave solutions for the high-dimensional Boussinesq
\begin{equation}\label{hbouss}
	u_{tt}-\Delta u +\de (\de u+|u|^{p-1}u)=0,\qquad(x,t)\in\rn\times\rr,
\end{equation}
where $1<p<2^\ast-1$ and $2^\ast=\frac{2n}{n-2}$ if $n>2$ and $2^\ast=\infty$ if $n=2$. This equation was originally derived by Boussinesq \cite{bouss}  from Euler's equations of motion for two-dimensional potential flow beneath a free surface by introducing suitable  approximations for small amplitude long waves. Equation \eqref{hbouss} was also derived in  the study of dynamics of thin inviscid layers with free surface, and besides, in the investigation of the nonlinear   string, the shape-memory alloys, the propagation of waves in elastic rods and in the continuum limit of lattice dynamics or
coupled electrical circuits (see \cite{esfahani} and the references therein).

In the one-dimensional case $n=1$, several authors have focused on the study of equation \eqref{hbouss}. The pioneering work can be found in the seminal paper by Bona and Sachs \cite{bonasachs}, where they employed Kato's abstract techniques to demonstrate that the Cauchy problem is locally and globally well-posed for small data.  {These} results were further improved in \cite{linares}, where global well-posedness of \eqref{hbouss} in the energy space   with  small data was established.
Recently, the inverse scattering transform and a Riemann-Hilbert approach for \eqref{hbouss} with quadratic nonlinearity were developed in \cite{chalen}. Various studies have  investigated  the stability of solitary waves of the generalized Boussinesq equation when $n=1$. It was proved in \cite{bonasachs}, using the theory developed in \cite{gss}, that for $1 < p < 5$ and $c^2>\frac{p-1}{4}$, the traveling wave is orbitally stable.  It  is known, in the one-dimensional case, \eqref{hbouss} possesses the traveling wave solution $u(x,t)=\ff(x+ct)$, where
\[
\ff(\xi)=\left(\frac{p+1}{2}\left(1-c^2\right)\right)^{\frac{1}{p-1}}{\rm sech}^{\frac{2}{p-1}}\left(\frac{p-1}{2}\sqrt{1-c^2}(\xi)\right)
\]
is the unique positive ground state of the elliptic equation
\[
-\ff''+(1-c^2)\ff= \ff ^{p-1}\ff,\qquad|c|<1.
\]
When $1 < p < 5$ and $c^2<\frac{p-1}{4}$, or $p \geq 5$ and $c^2< 1$, Liu proved the orbital instability of traveling waves in \cite{liu93}. He also demonstrated in \cite{liu95} that the traveling wave is strongly unstable by  the mechanism of blow-up if $c  = 0$. In \cite{lot}, Liu, Ohta, and Todorova showed that when $1 < p < \infty$ and $0 < 2(p+1)c^2 < p-1$, the traveling wave is strongly unstable. Recently,   Li,
Ohta, Wu, and Xue  in  \cite{lowx} proved the orbital instability in the degenerate case  $c^2=(p-1)/4$ with $1<p<5$.

In this paper,
we investigate the stability of the traveling wave solutions with nonzero wave speed of the generalized  Boussinesq equation \eqref{hbouss} in the case $n\geq2$ which has not been studied before. A few    numerical studies concerning the shape of traveling waves of \eqref{hbouss} when $p=n=2$  have been reported  in \cite{ChristouChristov} and \cite{Pel-Step}.

The issue of well-posedness of \eqref{hbouss} in the high-dimensional case was recently investigated in \cite{cgs}. More precisely, Chen,   Guo and  Shao in \cite{cgs} used the transformation $v=u+\ii  \mathfrak{B}^{-1}u_t$ to turn \eqref{hbouss} into
\begin{equation}\label{transformed-bq}
	\ii v_t-\mathfrak{B} v+\mathfrak{M}|\Re(v)|^{p-1}\Re(v)=0,
\end{equation}
where $\mathfrak{B}=\sqrt{-\Delta(I-\Delta)}$ and $\mathfrak{M}=\sqrt{\frac{-\Delta}{I-\Delta}}$. Thanks to the   Strichartz estimates, associated with the Schr\"{o}dinger-type group of \eqref{transformed-bq},  and the analysis of \cite{liu95},  they   obtained the following results \cite{cgs} 
\begin{theorem}  
	\label{local}
	The initial value problem
	\begin{equation}\label{hbouss-ivp}
		\begin{cases}
			u_{tt}-\Delta u +\de (\de u+|u|^{p-1}u)=0,\\
			(u(0),u_t(0))=(u_0,u_1),
		\end{cases}
	\end{equation}
	with $\vec u(t)=(u,(-\de)^{-\frac12}u_t)\in H^s(\rn)\times H^{s-1}(\rn)$  is locally well-posed for
	\[
	s=\begin{cases}
		0,&1<p<1+\frac 4n,\\
		1,&1+\frac 4n\leq p<2^\ast-1,\\
		\max\{0,\frac n2-\frac{2}{p-1}\},&p>1.
	\end{cases}
	\]
	Additionally, the following functionals are invariant under  the flow of \eqref{hbouss-ivp}:
	\[
	\E(u(t))=\frac12\int_\rn\abso{(-\de)^{-\frac12}u_t}^2+|u|^2+|\nabla u|^2\dd x-\frac{1}{p+1}\int_\rn|u|^{p+1}\dd x,
	\]
	\[
	\F(u(t))= \int_\rn((-\de )^{-\frac12}u_t)\nabla((-\de)^{-\frac12}u)\dd x.
	\]
	Moreover, if $0<\E(\vec u(0))<\frac{p-1}{2(p+1)}C_{GN}^{-\frac{2(p+1)}{p-1}}$  with $\vec u(0)=(u_0,u_1)$, and $\|u_0\|_\ho<C_{GN}^{-\frac{p+1}{p-1}}$, then \eqref{hbouss-ivp} is globally well-posed in $H^1(\rn)\times \dot{H}^{-1}(\rn)$, where $C_{GN}$ is the best constant for the Sobolev inequality 
	\[
	\|u\|_{L^{p+1}(\rn)}\leq C_{GN}\|\nabla u\|_\lt^{\frac{n(p-1)}{2(p+1)}}\|u\|_\lt^{1-\frac{n(p-1)}{2(p+1)}}.
	\]
	Furthermore, if $\vec u(0)\in H^1(\rn)\times \dot{H}^{-1}(\rn)$ 
	 such that 
	\[
	\E(\vec u(0))<\frac{p-1}{2(p+1)}C_{GN}^{-\frac{2(p+1)}{p-1}}\quad\text{and}\quad\|u_0\|_\ho>C_{GN}^{-\frac{p+1}{p-1}},
	\]    
	then the local solution $u(t)\in\ho$ blows up in finite time.
\end{theorem}
\begin{remark}\label{gn-remark-0}
	It is known from \cite{liu95} that the constant $C_{GN}$ can be  represented  in terms of the (unique stationary) ground state of \eqref{hbouss}. More precisely,
	\[
	C_{GN}=\|\ff_0\|_{H^1(\rn)}^{-\frac{p-1}{p+1}}=\|\ff_0\|_{L^{p+1}(\rn)}^{-\frac{p-1}
		{2}}
	=\left(\frac{2(p+1)}{n+2-p(n-2)}\right)^{-\frac{p-1}{2(p+1)}}\|\ff_0\|_\lt^{-\frac{p-1}{p+1}},
	\]
	where $\ff_0$ is the unique positive radial solution of
	\begin{equation}\label{c=0}
		\ff-\Delta\ff=\ff^p.
	\end{equation}
\end{remark}


Here, we are interested in studying the dynamics of (angular) traveling waves of \eqref{hbouss} with nonzero wave speed.  Indeed, we are looking for a solution of the form $u(x,t)=\ff(x+\vec ct)$, which is a traveling wave solution of \eqref{hbouss}. So the profile $\ff$ must satisfy 
\begin{equation}\label{local-eq}
	(\vec c\cdot\nabla)^2\ff-\de \ff+\de\paar{\de\ff+\abso{\ff}^{p-1}\ff}=0
\end{equation}
or, equivalently
\begin{equation}\label{traveling}
	(I+\L_\vc)\ff-\de\ff=\ff^p,
\end{equation}
where $$\L_\vc f=\frac{ (\vec c\cdot\nabla)^2}{-\Delta}f=- {\rm p.v.}\;f\ast (\vec c\cdot\nabla)^2N_n,$$ and $N_n(x)$ is the Laplacian fundamental solution.  Since \eqref{hbouss} is invariant under orthogonal transformations, we may choose $\vec c=(c,0,\cdots,0)=c\vec e_1\in\rn$ with $|c|<1$. We first show the existence of solutions of \eqref{traveling}. This is obtained by applying the concentration-compactness principle due to the boundedness of operator $\L_\vc$  in the Lebegsue spaces (see Theorem \ref{exits}).  We refer the reader to \cite{haj-cho-ozawa,haj-2012,haj-2013,haj-st-2004,haj-chang-song,hpl,hS} and the references therein, which discuss the minimization under constraint (and stability) for similar equations. 

In spite of having a bounded operator $\L_\vc$ in \eqref{traveling}, the behavior of solutions of \eqref{second-der} is different from the stationary case $\vc=0$. First of all, it follows, for example from the results of Berestycki and Lions \cite{ber-lions},  that \eqref{traveling}  with $\vc=0$  possesses  a unique positive radially symmetric ground state solution. Moreover, the solution $\ff$ of \eqref{traveling} in the zero-speed $\vc=0$ (see \eqref{c=0}) is smooth and decays exponentially. More precisely, it holds for some $\delta>0$ that
\[
\left|\partial^\al\ff(x)\right|\leq C_\al\exp(-\delta|x|),\qquad\al\in\N^n,\;|\al|\leq2.
\]
In the nonzero speed case, we show in Theorem \ref{decay-regularity} that the traveling waves decay algebraically. To prove this, we  observe that the kernel of \eqref{traveling} (see \eqref{traveling-conv}) behaves like one of  the Gross-Pitaevskii equation.
Secondly, contrary to the zero-speed case, the operator $\L_\vc$ in the case $\vc\neq0$ is not radial, so to show the symmetry of traveling waves, we use an argument developed in \cite{blss} which is based on the Steiner type rearrangements in Fourier space (see Theorem \ref{symmetry}). Thirdly, the traveling waves of \eqref{hbouss}, contrary to the case $\vc=0$, are not positive, and together with lack of information about the uniqueness of \eqref{traveling} make the analysis of \eqref{traveling} difficult.

The next aim is to investigate the orbital stability of traveling waves in \eqref{traveling}, while bearing in mind the local/global well-posedness ensured by Theorem \ref{local} (refer to Definition \ref{dfn-1}). To proceed, we first find some variational characterizations of the traveling waves by transforming \eqref{traveling} into a system of equations possessing the Hamiltonian structure. This helps us to show that \eqref{traveling} has ground states. Next, by using the variational properties of ground states, we establish the stability of the set of all ground states of \eqref{traveling}, provided that the Lyapunov function is convex with respect to $|\vc|$ (see Theorem \ref{stability-t-1}).  The uniqueness of the ground state of \eqref{traveling} is unknown due to the presence of $\L_\vc$, so the aforementioned result is somehow weaker than the usual orbital stability. On the other hand, the classical theory of \cite{gss} seems not to be easy to be applicable due to the lack of spectral information of the linearized operator of \eqref{traveling} in the case $\vc\neq0$. To overcome this difficulty, in Section \ref{section-var-char} we use the perturbation theory and show that the ground states of \eqref{traveling} converge to the unique ground state $\ff_0$ of \eqref{c=0}, and the spectral behavior of $\ff_0$ can be inherited to the ground states $\ff_\vc$ of \eqref{traveling} provided $\vc$ is sufficiently small (Theorem \ref{stab-theo}).

Another difficulty in the study of the orbital stability of traveling waves is the lack of scaling.  As it is seen e.g. in \cite{bss}, an appropriate scaling transfers the role of wave speed from traveling wave solely into the Lyapanuv function and one can analyze the orbital (in)stability of traveling waves by determining the convexity of the Lyapanuv function (see e.g. \cite{gss}). To overcome this difficulty in the case of nonzero wave speed, we apply the geometric and variational approaches  developed by \cite{ribeiro} in Section \ref{section-orb-ins}.

Our other result is related to the strong instability of traveling waves of \eqref{traveling}. The question is to find the conditions under which the traveling waves are unstable due to the mechanism of blow-up. To proceed, in Section \ref{section-str-ins}, we first demonstrate that the ground states of \eqref{traveling} possess the mountain-pass geometry. Then, we follow the ideas presented in \cite{lot} and decompose $H^1$ into two submanifolds that are invariant under the flow of the Cauchy problem associated with \eqref{hbouss}.  This decomposition establishes conditions under which the local solution of \eqref{hbouss} either exhibits global behavior or undergoes a finite-time blow-up, thereby enhancing the results presented in Theorem \ref{local} (see Corollary \ref{blow-up-cor}).   The aforementioned submanifolds depend strongly on the wave speed $\vc$, so our results  improve  upon those obtained in \cite{cgs}.

We finish this section by remarking that the approaches used in this paper can be applied to the study of angular traveling waves of the high dimensional improved Boussinesq equation (see \cite{ChrisC,ChristovTodorovChristou})
\[
(1-\beta_1\Delta )u_{tt}=\Delta   ( u-\al u^2-\beta_2\Delta u),
\] 
where $u$ is the surface elevation, $\beta_1,\beta_2>0$ are two dispersion coefficients, and $\alpha$ is an amplitude parameter.

\vspace{2mm}
Throughout this paper, we assume that $p+1=k/m$, where $k$ is even and $m$ is odd. 

\section{Existence,  variational characterizations, and stability}\label{section-var-char}
In this section, we show the existence, and variational properties, and orbital stability of the traveling waves of \eqref{hbouss}.   To formulate our results, we define the following natural minimization problem
\[
m(\vec c,\lam)=\inf\sett{J(u),\;u\in H^1(\rn),\, K(u)=\lam },
\]
where 
\[
J(u)=\frac12\int_\rn
\left(\abso{\nabla u}^2+\abso{u}^2-\abso{\L_\vc^{\frac12}u}^2\right)\dd x
\]
and $K(u)=\frac{1}{p+1}\|u\|_{L^{p+1}(\rn)}^{p+1}$. By homogeneity, it follows that
\[
m(\vec c)\equiv m(\vec c,1)=\inf\left\{\frac{J(u)}{K^{\frac{2}{p+1}}(u)},\;u\in H^1(\rn)\setminus\{0\}\right\}.
\]
It follows from the concentration-compactness principle \cite{lions} that the above minimization problem has a solution for each $\lambda > 0$, provided $\vc$ is chosen so that the functional $J$ is coercive. This coerciveness condition holds because $|\vec{c}\cdot\nabla| \leq |\vec{c}||\nabla|$  when $|\vec{c}| < 1$. In the context of this problem, a sequence ${\ff_k}\subset H^1(\mathbb{R}^n)$ is considered a minimizing sequence for $J$ if $J(\ff_k)\to m(c,\lambda)$ as $k\to\infty$ and $K(\ff_k)=\lambda$ for all $k$.

It is important to note that $\L_\vc f = (\vc\cdot\vec{R})\ast(\vc\cdot \vec{R}) f$, where $\vec{R}=(R_1,\cdots,R_n)$, and $R_j$, $1\leq j\leq n$, represents the Riesz transform defined by $\hat{R}_j(\xi)=-\ii\frac{\xi_j}{|\xi|}$. For $n>1$, it is well-known that $R_j(x)=\frac{1}{\pi\omega_{n-1}}{\rm p.v.}$$\frac{x_j}{|x|^{n+1}}$. The $L^p$-boundedness of $R_j$ implies that $I+\L_\vc$ is homogeneous of zeroth order and bounded on $L^p(\mathbb{R}^n)$ for any $1<p<\infty$. Moreover, $I+\L_\vc\geq0$ if $|\vc|<1$.
 	Hence, by employing  
	\begin{equation}\label{vector-gn}
		\|u\|_{L^{p+1}(\rn)}\leq C_{\vc}
		\|\nabla u\|_\lt^{\frac{n(p-1)}{2(p+1)}}
		\left(\int_\rn\left(|u|^2-\abso{\L_\vc^{\frac12}u}^2\right)\dd x \right)^{\frac{n+2-p(n-2)}{4(p+1)}},
	\end{equation}
	 and the same lines as in \cite{lions}, the following theorem is established. 

\begin{theorem}\label{exits}
	Assume that  $1<p<\frac{n+2}{n-2}$ if $n>2$ and $p>1$ if $n=2$. Let $|\vec c|<1$, and $\{\ff_k\}$ be a minimizing sequence for $m(\vc,\lam)$. Then there exists a subsequence $\{\ff_{k_j}\}$ of $\{\ff_k\}$, $\{y_j\}\subset\rn$ and $\ff\in H^1(\rn)$ such that $\ff_{k_j}\to\ff$ in $H^1(\rn)$ and  $\ff$ is a minimizer of $J$ subject to the constraint $K(\ff)=\lam$.
\end{theorem}

\begin{remark}\label{gn-remark-1}
	It is worth noting that $C_\vc$ in \eqref{vector-gn} can be estimated with $C_{GN}$ in Remark \ref{gn-remark-0}. More precisely, one can observe from
	\[
	(1-|\vc|^2)\|u\|_\lt^2\leq \int_\rn\left(|u|^2-\abso{\L_\vc^{\frac12}u}^2\right)\dd x
	\leq
	\|u\|_\lt^2
	\]
	and the variational characterization (see \cite{wein-1983})
	\[
	C_\vc^{-(p+1)}=\inf_{u\in \ho\setminus\{0\}}
	\frac{\|\nabla u\|_\lt^{\frac{n(p-1)}{2}}
		\left(\|u\|_\lt^2-\|\L_\vc^{\frac12}u\|_\lt \right)^{\frac{n+2-p(n-2)}{4}}}
	{(p+1)K(u)}
	\]
	that
	\begin{equation}\label{gn-equiv}
		(1-|\vc|^2)^{  \frac{p(n+
				2)-n-2}{4(p+1)}}
		\geq
		\frac{ C_\vc}{C_{GN}} 
		\geq
		1.
	\end{equation}
	By an appropriate scaling, one can show that
	\begin{equation} 
		C_{\vc}	\sim	(1-|\vc|^2)^{  \frac{p(n+
				2)-n-2}{4(p+1)}}	C_{GN}
		.
	\end{equation}
\end{remark}

Suppose that $\ff$ is a constrained minimizer for $m(\vec c,\lam)$ guaranteed by Theorem \ref{exits}. Then, $\ff$ satisfies 
\[
J'(\ff)=\ell K'(\ff)
\]
for some Lagrange multiplier $\ell\in\rr$. It then follows that 
\[
J'(\phi)=  K'(\phi),
\] 
where $\phi=\ell^{\frac{1}{p-1}}\ff$. Thus $\phi$ is a solution of \eqref{traveling} and achieves the minimum $m(\vec c)$.

By multiplying \eqref{traveling} by $\ff$ and integrating over $\rn$, we obtain that
\[
2J(\ff)=(p+1)K(\ff);
\]
so that
\[ 
(p+1)K(\ff)=2J(\ff)=2\left(\frac{2}{p+1}\right)^{\frac{2}{p-1}}(m(\vec c))^{\frac{p+1}{p-1}}.
\]  Solutions of \eqref{traveling} achieving $m(\vec c)$  are  called the ground states. The set of all ground states of \eqref{traveling} is denoted by
\[
G(\vec c)=\left\{u\in H^1(\rn),\;(p+1)K(\ff)=2J(\ff)=2\left(\frac{2}{p+1}\right)^{\frac{2}{p-1}}(m(\vec c))^{\frac{p+1}{p-1}}\right\}.
\]

We also note when $\vec c \neq\vec 0$ that $P(\xi)=-\frac{(\vec c\cdot\xi)^2}{|\xi|^2}$ is the   Fourier multiplier of $\L_{\vec c}$, and it is cylindrically symmetric with respect to  the  direction $e\in \mathbb{S}^{n-1}$, where $\vec c=\ell e$ for some $\ell\in\rr$. Recall that a function $f:\rn\to\rr$
is cylindrically symmetric with respect to a direction $e\in \mathbb{S}^{n-1}$ if we have $f(\mathsf{R}y)=f(y)$ for almost every $y\in\rn$ and all $\mathsf{R}\in O(n)$ with $\mathsf{R}(e)=e$. For such functions, we use the notation $f(y)=f(y_\|,|y_\bot|)$, where   $y = y_\|+ y_\bot $ and $y_\bot$ is perpendicular to $e$.  Let $f^{\ast_e}:\rn\to\rr^+$ denote  the Steiner symmetrization in $(n-1)$-codimensions with
respect to the direction $e$. This symmetrization is obtained by symmetric-decreasing rearrangements in
$(n-1)$-dimensional planes perpendicular to $e$.

\begin{theorem}\label{symmetry}
	Let $\vec{c}=\ell e$ with $e\in \mathbb{S}^{n-1}$ and $\ell\in\mathbb{R}$. Assume that $p$ is an odd number. Then any ground state $\ff$ of \eqref{traveling} has the form $\ff^{\sharp_e}(x+x_0)$ for some $x_0\in\mathbb{R}^n$, where $\ff^{\sharp_e}:=((\hat{\ff})^{\ast_e})^\vee$. Moreover, $\ff(x)=\ff(\mathsf{R}x)$ for all $\mathsf{R}\in O(n)$ such that $\mathsf{R}\vec{c}=\vec{c}$.
\end{theorem}
\begin{proof}
	The proof is followed  by noting that $\L_{\vec c}-\Delta$  satisfies the assumptions of Theorem 2 in \cite{blss}. Indeed,  one observes that  $|\xi|^2-|\vec c|^2\leq|\xi|^2+P(\xi)\leq |\xi|^2$, and $|\xi|^2+P(\xi)$ is cylindrically
	symmetric with respect to $\vec c$  and $|\xi_\bot|\mapsto |\xi|^2+P(\xi)$ is strictly increasing. This can be easily checked by using a suitable transformation to get $\vec c=(c,0,\cdots,0)$.
\end{proof}
\begin{remark}\label{cpt-embed-rem}
 	It is worth noting that if   $G$ is   a subgroup  of $O(n)$ with $n\geq2$, we can define the space
	\[
	H^1_G=\sett{u\in H^1(\rn),\; u\;\text{is invariant with respect to}\; G}.
	\]
	 In particular, $H_G^1$ is the subspace of all radial  functions of $H^1(\rn)$ if $G=SO(n)$. Corollary 2 in \cite{Skrzypczak} shows that if the orbit $H\cdot x$ is infinite for any $x\in\mathbb{S}^{n-1}$, then $H^1_G$ is compactly embedded into $L^p(\mathbb{R}^n)$ for any $p\in(2,2^\ast)$. Hence, if $J$ is well-defined on $H_G^1$ for some $\vc$, then all minimizing sequences of $m(\vec{c})$  associated with $H_G^1$, are relatively compact in $\lt$. Consequently, \eqref{traveling} has a ground state in $H_G^1$. When $\vc=0$, it is known that that \eqref{traveling} has a radial ground state.
 
\end{remark}

\begin{remark}
	The numerical computations validate the result of Theorem \ref{symmetry}. As depicted in Figure \ref{fig-2}, the contour sets of solutions of \eqref{traveling} exhibit elliptic shapes. Furthermore, the ellipses become narrower (elongated along the major axis) as $\vc$ approaches one, whereas they shrink into circles as $\vc$ approaches zero.
\end{remark}

Our next aim is to investigate the decay behavior of traveling waves.  Notice that   \eqref{traveling} can be re-written in the convolution form
\begin{equation}\label{traveling-conv}
	\ff=\K_{\vec{c}}\ast\ff^p,
\end{equation}
where
\[
\hat{\K}_{\vec{c}}(\xi)=\frac{1}{1+|\xi|^2-\frac{1}{|\xi|^2}(\vec c\cdot\xi)^2}.
\]
The rational function $\hat{\K}_{\vec{c}}$ is only singular at the origin, where the singularity is of the form $O(|\xi|^{-a})$ for some $a>0$. The following result regarding the decay behavior of 
$\hat{\K}_{\vec{c}}$
follows from \cite{gravej2004}, where integrability estimates of the derivatives of 
$\hat{\K}_{\vec{c}}$
are demonstrated. These estimates provide the algebraic decay of 
$\hat{\K}_{\vec{c}}$
using the inverse Fourier transform formula. 
Indeed, one can observe that the asymptotic properties of $\K_{\vec{c}}$ are
mainly determined  by the behavior of  $\hat{\K}_{\vec{c}}$  near to the origin.     More precisely, 
\[
\hat{\K}_{\vec{c}} (\xi)\sim\frac{|\xi|^2}{|\xi|^2 -(\vc\cdot\xi)^2}.
\] 
Moreover, by using some stationary phase estimates, we observe for $\vc=(c,0,\cdots,0)$ that
\[
\hat{\K}_{\vec{c}} (\xi)=\hat{\K}_{\vec{c}} (\xi_1,\xi_\bot)
\sim\frac{1}{1-c^2}\widehat{R_{1,1}}(\sqrt{1-c^2}\xi_1,\xi_\bot)+\frac{1}{2}\sum_{j=1}^{n}
\widehat{R_{j,j}}(\sqrt{1-c^2}\xi_1,\xi_\bot),
\]
where
\[
R_{j,j}(x)=\frac{\Gamma\left(\frac n2\right)}{2\pi^{\frac n2}}
{\rm p.v.}\frac{|x|^2-nx_j^2}{|x|^{n+2}}.
\]
Actually, it is clear that $R_{j,j}=-R_j\ast R_j$.
\begin{lemma}\label{decay-kern}
	Let $\kappa\in(n-2,n]$. Then there exists $C_\kappa>0$ such that\[
	\left|\K_{\vec{c}}(x)\right|\leq C_\kappa|x|^{-\kappa}
	\]
	for all $x\in\rn$. In particular, $\K_{\vec{c}}\in L^r(\rn)$ for any $1<r<\frac{n}{n-2}$, and $\nabla \K_{\vec{c}}\in L^r(\rn)$ for any $1\leq r<\frac{n}{n-1}$. Moreover, for any $n\in\N$,
	\[
	\sup_{x\in\rn}|x|^{\kappa+n}\left|\frac{\d^n}{\d x^n}\K_{\vec{c}}\right| <+\infty.
	\]
\end{lemma}

\begin{remark}
	We should remark that if one uses a suitable transformation to assume that $\vc=(c,0,\cdots,0)$, then by following the same ideas of \cite{gravej2004}, one can show (see \cite{gravej2005}) that
	\[
	R^n\K_\vc\ast f(Rx)\to
	\frac{\Gamma\left(\frac n2\right)c^2(1-c^2)^\frac{n-3}{2}(1-c^2+(c^2-n)x_1^2)}{2\pi^{\frac n2}(1-c^2+c^2x_1^2)^\frac{n+2}{2}}
	\int_\rn f(x)\dd x
	\]
	as $R\to\infty$.
\end{remark}
\begin{theorem}\label{decay-regularity}
	Let $\ff\in\ho$ be a traveling wave of \eqref{traveling} with $|\vc|<1$. Then $\ff\in W^{2,q}(\rn)$ for any $1<q<\infty$. Moreover,    $ |x|^{\kappa}\ff\in L^\infty(\rn)$ for any $\kappa\in(n-2,n]$. If $s\mapsto s^p$ is smooth for any $s>0$, then $\ff\in H^\infty(\rn)$. 
\end{theorem} 
\begin{proof}
	The proof of regularity  relies on the Lizorkin theorem \cite{Lizorkin} and standard arguments involving the Riesz operators. This is achieved by considering the fact that the kernels $\K_{\vec{c}}$, $\partial_{x_j}\K_{\vec{c}}$, and $\partial_{x_i}\partial_{x_j}\K_{\vec{c}}$ (with $1\leq i,j\leq n$) act as $L^r(\mathbb{R}^n)$-multipliers for $1<r<\infty$.

	The algebraic decay of the traveling wave is deduced from Lemma \ref{decay-kern} and argument of \cite{bonali}, so we omit the details.
\end{proof}

\begin{remark}
	In spite of being unable to prove it, numerical computations shown in Figures \ref{fig-1} and \ref{fig-3} indicate that traveling waves of \eqref{traveling} are not positive when $\vc\neq0$. This situation is different from the case where $\vc=0.$
\end{remark}

\begin{figure}[ht]
	\begin{center}
		\scalebox{0.28}{\includegraphics{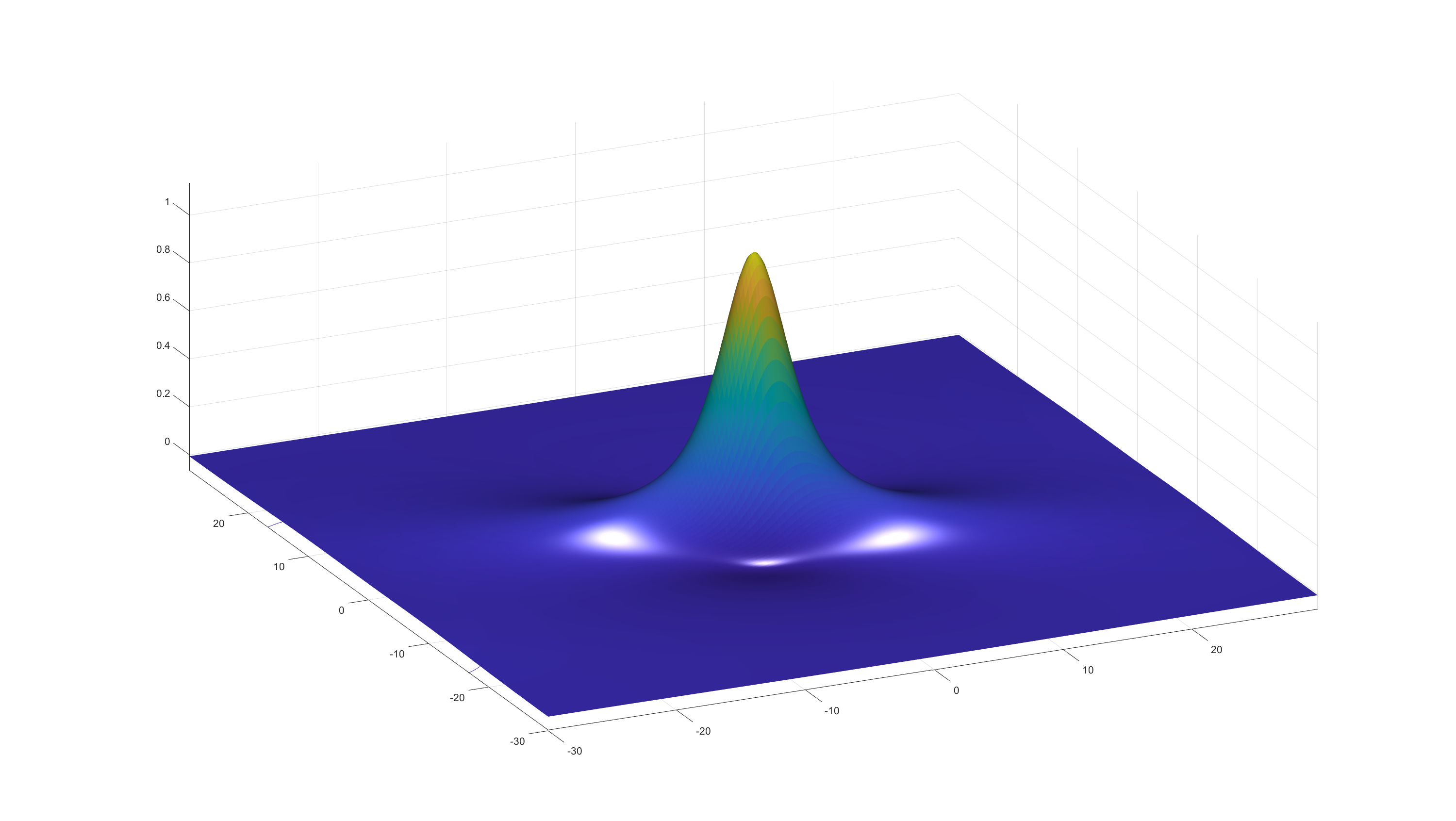}   	}
		\caption{The traveling wave of \eqref{traveling} with $c=0.8$.  } \label{fig-1} 
\end{center}\end{figure}

\begin{figure}[ht]
	\begin{center}
		\scalebox{0.28}{\includegraphics{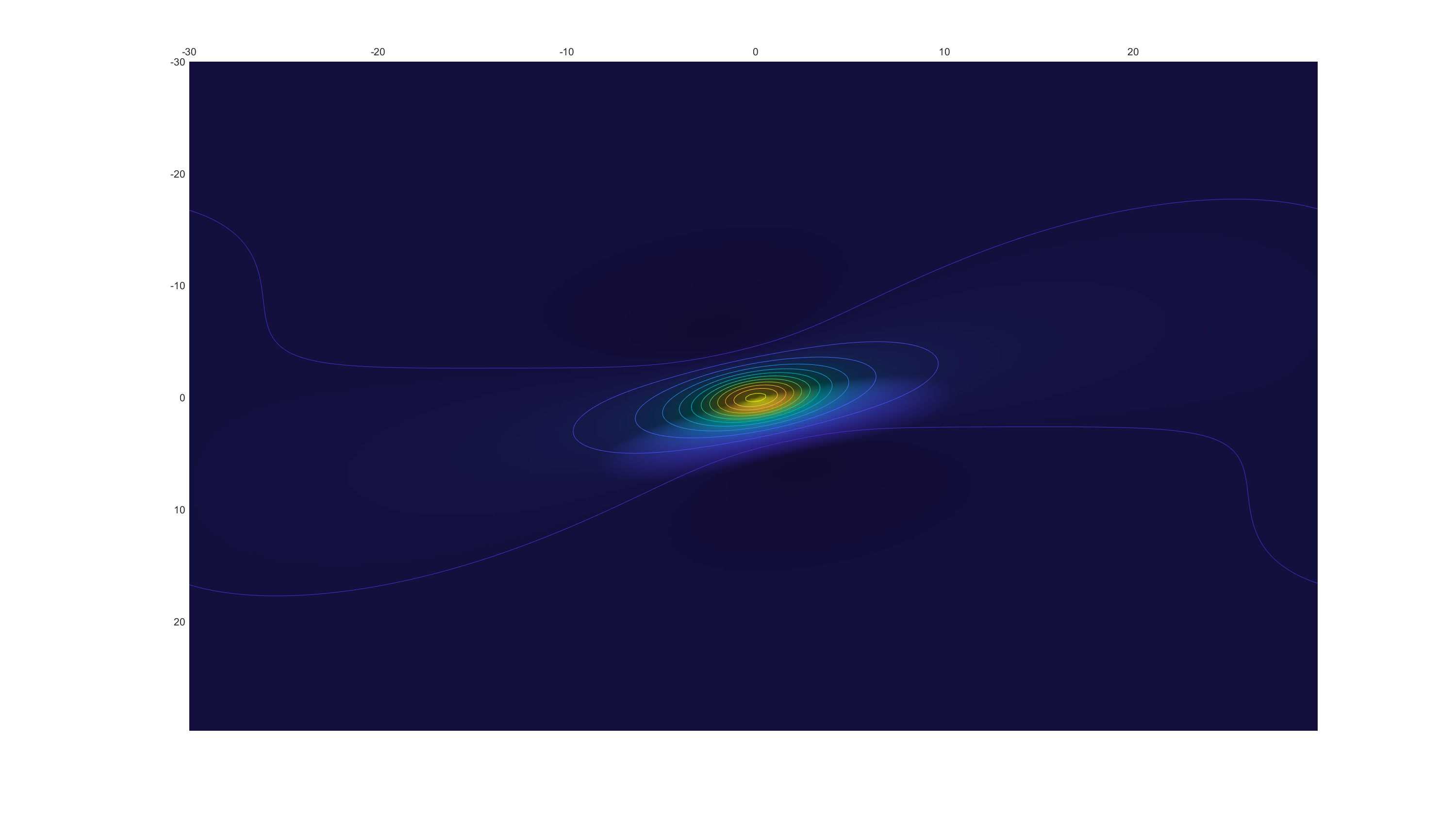}   	}
		\caption{The contours of solutions of \eqref{traveling} with $c=0.8$.  }\label{fig-2} 
\end{center}\end{figure}

\begin{figure}[ht]
	\begin{center}
		\scalebox{0.18}{\includegraphics{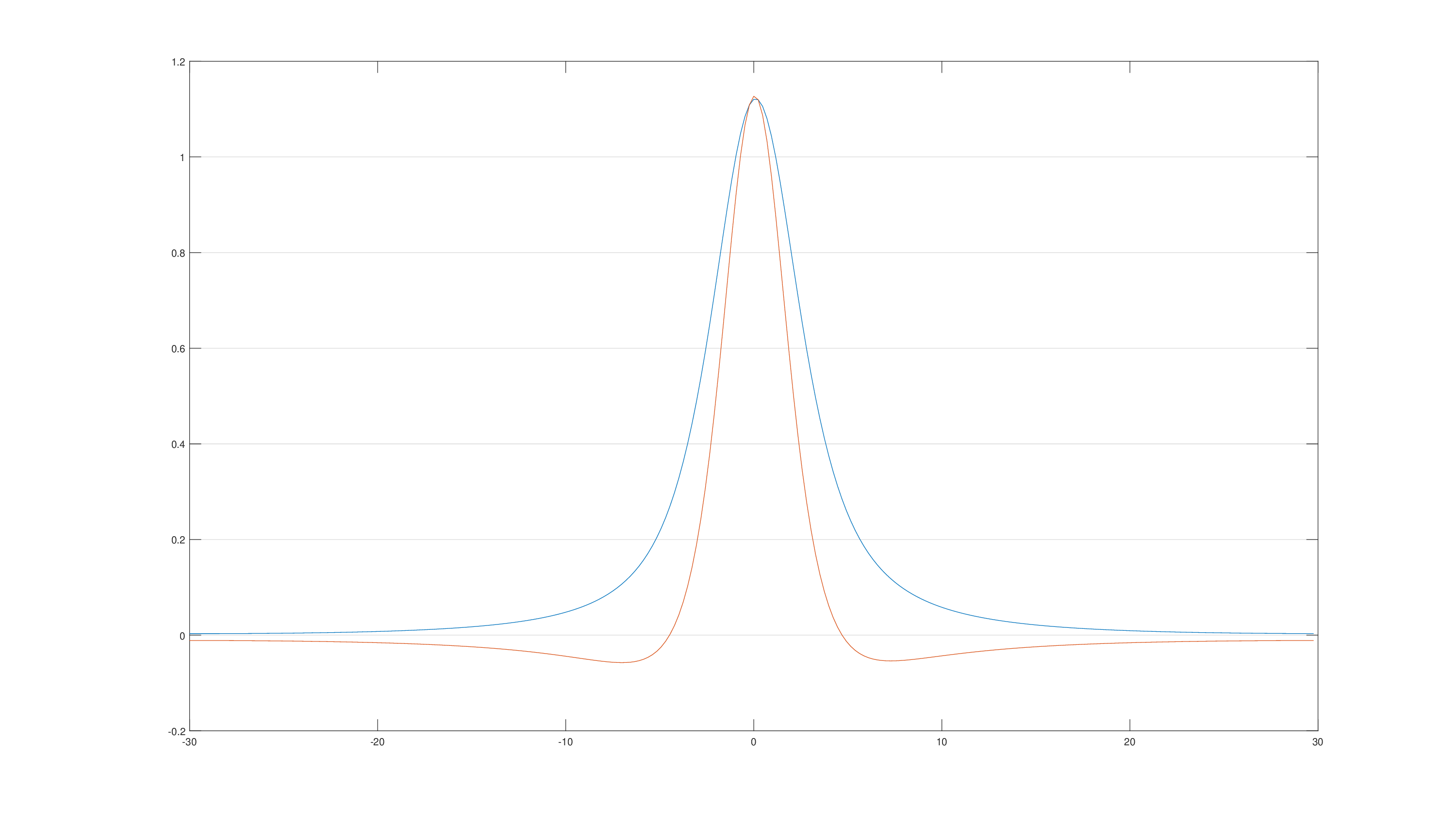}  \includegraphics{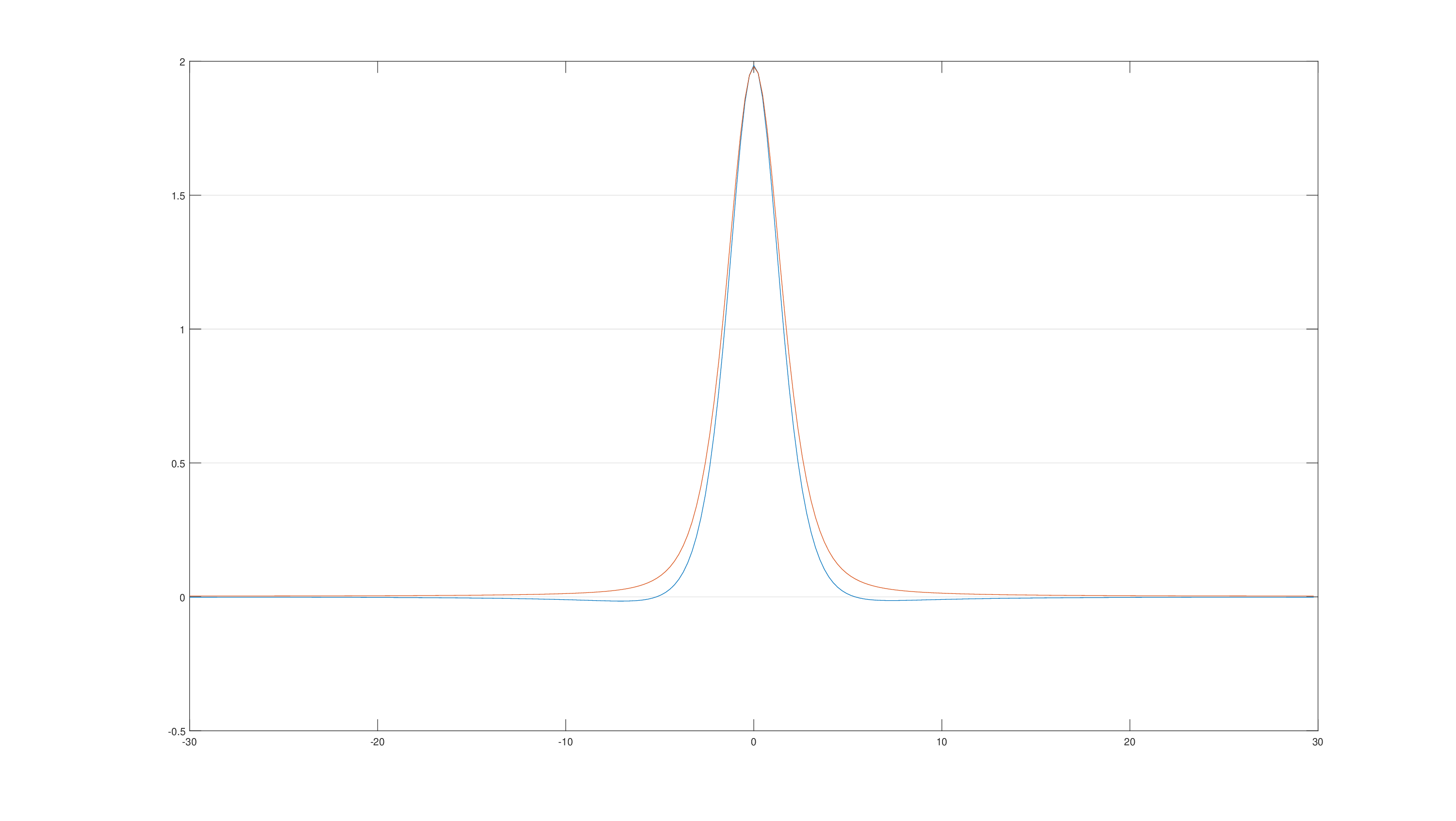}}
		\caption{The projections of traveling wave with $c=0.8$ (left) and $c=0.2$ (right) on $XY$ and $XZ$ planes.  }\label{fig-3} 
\end{center}\end{figure}

\begin{remark}
	It is worth noting that by using classical calculations, one can find the following Pohozaev identities:
	\begin{equation}\label{pohoz}
		\begin{split}
			&\frac2n\|\nabla\ff\|_\lt^2=\frac{p-1}{p+1}\|\ff\|_{L^{p+1}(\rn)}^{p+1},\\
			&\frac{n+2-p(n-2)}{2(p+1)}
			\|\ff\|_{L^{p+1}(\rn)}^{p+1}=\| \ff\|_\lt^2-\norm{\ltt\ff}_\lt^2,
			\\&
			\frac{n+2-p(n-2)}{n(p-1)}\|\nabla\ff\|_\lt^2
			=\|\ff\|_\lt^2-\norm{\ltt\ff}_\lt^2 .
		\end{split}
	\end{equation}
	These identities will be helpful in our instability analysis. 
\end{remark}
To study the stability of traveling waves of \eqref{traveling}, we observe that \eqref{hbouss} can be equivalently represented as a system:
\begin{equation}\label{system}
	\begin{cases}
		&	u_t=(-\de)^{\frac12}v,\\
		&v_t=(-\de)^{\frac12}(u-\de u-|u|^{p-1}u). 
	\end{cases}
\end{equation}
In this form, the momentum $\F$ and the energy $\E$ are expressed as:
\[
\E(\vec u)=\frac12\int_\rn|v|^2+|u|^2+|\nabla u|^2\dd x-\frac{1}{p+1}\int_\rn|u|^{p+1}\dd x,
\quad
\F(\vec u)= \int_\rn v\nabla((-\de)^{-\frac12}u)\dd x.
\]
Additionally, system \eqref{system}   possesses  the Hamiltonian structure 
$\vec u_t=\njj \E'(t)$, where $\vec{u}=(u,v)^T$ and
\[
\njj=\begin{bmatrix}
	0&(-\de)^{\frac12}\\
	-(-\de)^{\frac12}&0
\end{bmatrix}.
\]

We next turn our attention to establishing an  alternative  variational characterization of traveling waves.
Define, for $\vec\ff=(\ff,\psi)$, the action functional 
\begin{equation}\label{vect-S}
	\begin{split}
		S_{\vec c}(\vec \ff)&=\E(\vec\ff)+\vc\cdot\F(\vec\ff)
		=\ns  (\ff)+\frac12\norm{\L_\vc^{\frac12}\ff+\psi}_\lt^2,
	\end{split}
\end{equation} 
where $\ns=J-K$.
Given a traveling wave(ground state) $\ff\in G(\vec c)$, the corresponding traveling wave solution of \eqref{system}  takes the form $(\ff(x+\vec ct),-\ltt\ff(x+\vec ct))$, so we define
$\vec\ff=(\ff ,-\L_\vc^{\frac12}\ff)$ and set
\[
\ngg(\vec c)=\left\{\vec \ff=(\ff ,-\L_\vc^{\frac12}\ff)\in \x,\;\ff\in G(\vec c)\right\},
\]
where $\x=H^1(\rn)\times\lt$.
It follows from this definition that any function $\vec\ff=(\ff,\psi)\in\ngg(\vec c)$   is indeed a critical point of $S_\vc$, viz.

\begin{equation}\label{criticalpoint}
	S_\vc'(\vec\ff)=\left(\begin{array}{cc}
	 \ff-\de\ff-\ff^p-\L_\vc^{\frac12}\psi \\
		\psi+\L_\vc^{\frac12}\ff
	\end{array}\right)=\vec 0.
\end{equation}

This also implies for any $\vec\ff\in\ngg(\vec c)$ that $P(\vec \ff)=0$, where $P(\vec\ff)=\la S'_{\vec c}(\vec\ff),\vec \ff\ra$. 	It is easy to see that $\vec\ff=(\ff,\psi)\in\x$ is a critical point of $S_{\vec c}$ if and only if $\ff$ is a critical point of $\ns$ and $\psi=-\L_\vc^{\frac12}\ff$.
\begin{theorem}\label{theorem-variational-2}
	Let $n\geq2$. Then $\vec\ff\in\ngg(\vec c)$ if and only if $\vec\ff\in N$ and 
	\[
	S_{\vec c}(\vec\ff)=\inf\sett{S_{\vec c}(\vec w),\;\vec w\in N},
	\]
	where
	\[
	N=\sett{\vec w\in \x\setminus\{0\},\;P(\vec w)=0}.
	\]
\end{theorem}
\begin{proof}
	First we note for $\vec\ff=(\ff,\psi)$ that $P=2J-(p+1)K+2q$, where $q(\vec\ff)=\frac12\norm{\L_\vc^{\frac12}\ff+\psi}_\lt^2$. Hence, if $\vec u=(u,v)\in N$, then
	\[
	S(\vec u)=\frac{p-1}{2}K(u)=-\frac{p-1}{p+1}J(u)+\frac{p-1}{p+1}q(\vec u).
	\]
	Now assume that $\vec\ff=(\ff,\psi)\in \ngg(\vc)$. Then $P(\vec\ff)=0$, so $\vec\ff\in N$ and $S(\vec\ff)=\frac{p-1}{p+1}J(\ff)$. Thence, by definition, we have for any $\vec u=(u,v)\in N$ that
	\[
	\frac{J(u)}{K^{\frac{2}{p+1}}(u)}\geq \frac{J(\ff)}{K^{\frac{2}{p+1}}(\ff)}.
	\]
	Thus,
	\[
	S(\vec u)\geq \frac{p-1}{p+1}J(u)
	\geq
	\frac{K^{\frac{2}{p+1}}(u)}{K^{\frac{2}{p+1}}(\ff)}\frac{p-1}{p+1}J(\ff)
	\geq S(\vec\ff).
	\]
	Conversely, if $\vec u=(u,v)$ minimizes $S$ over $N$, then  $S'(\vec u)=\theta P'(\vec u)$ for some Lagrange multiplier $\theta\in\rr$. Hence, 
	\[
	0=P(\vec u)=\la S'(\vec u),\vec u\ra=\theta\la P'(\vec u),\vec u\ra.
	\]
	Due to the fact
	\[
	\la P'(\vec u),\vec u\ra=2P(\vec u)-(p^2-1)K(u)=(1-p^2)K(u)\neq0,
	\]
	it is deduced that $\theta=0$, so that $S'(\vec u)=0$  implies that 
	\[
	0=\left\la S'(\vec u),(0,v+L^{\frac{1}{2}}_\vc u )\right\ra=2q(\vec u).
	\]
	This shows that $v=-L^{\frac{1}{2}}_\vc u$. To show $u\in G(\vc)$  and consequently $\vec u\in\ngg(\vc)$, it is enough to prove that $J(u)\leq J(u_1)$ for all $u_1$ satisfying $K(u)=K(u_1)$. For any such $u_1$, we set $\vec u_1=(u_1,-L^{\frac{1}{2}}_\vc u_1)$. Hence, we observe that
	\[
	P(s\vec u_1)=2s^2J(u_1)-(p+1)s^{p+1}K(u_1)=0
	\]
	for $s=\left(\frac{2J(u_1)}{(p+1)K(u_1)}\right)^{\frac{1}{p-1}}$. This implies that $S(\vec u)\leq S(s\vec u_1)$, and thereby
	\[
	J(u)\leq J(su_1).
	\]
	The facts $K(u)=K(u_1)$ and $2J(u)=(p+1)K(u)$ reveal that
	\[
	J(u)\leq J(su_1)=s^2J(u_1)=
	\left(\frac{2J(u_1)}{(p+1)K(u_1)}\right)^{\frac{2}{p-1}}
	J(u_1)\leq
	J(u_1).
	\]
	Therefore, $J(u)\leq J(u_1)$, and the proof is complete.
\end{proof}

\begin{remark}\label{remark-vect}
	Analogous to Theorem \ref{theorem-variational-2}, it is easy to see that if  $\ff\in H^1(\rn)$ satisfies $\PP(\ff):=\la\ns'(\ff),\ff\ra=0$  and 
	
	\begin{equation}\label{minim-1}
		\inf\{\ns(u),\;u\in H^1(\rn)\setminus\{0\},\;\PP(u)=0\}=\ns(\ff),
	\end{equation}
	then $\ff$ is a weak solution of \eqref{traveling}. Moreover, by applying the concentration-compactness principle, there exists a minimizer for \eqref{minim-1}. Moreover, 
	\[
	\begin{split}
		G_1(\vc)&:= \sett{\ff\in\ho\setminus\{0\},\;\ns'(\ff)=0,\;\ns(\ff)\leq\ns(u),\,\mbox{for all}\, u\in   \ho\setminus\{0\}\,\mbox{such that}\,\ns'(u)=0}\\&
		=\left\{\ff\in\ho,\;\ns(\ff)=m_1(\vc),\;\PP(\ff)=0\right\},
	\end{split}
	\]
	where $$m_1(\vc)=\inf\sett{\ns(u),\;u\in\ho\setminus\{0\},\;\PP(u)=0}.$$ Furthermore, $m_1(\vc)=m(\vc)$ and $G(\vc)=G_1(\vc)$. This claim is proved  by following the same lines of proof of Theorem \ref{theorem-variational-2} after replacing $S$ and $P$ with $\ns$ and $\PP$.
\end{remark}

\begin{definition}\label{dfn-1}
	The traveling wave   $\vec\ff$ of \eqref{system} is said to be
	orbitally stable if for any $\epsilon$, there exists $\delta  > 0$ such that if $\|\vec u_0-\vec\ff\|_\x<\delta$,
	then the solution $\vec u(t)\in \x$ of \eqref{system} with $\vec u(0)=\vec u_0$ exists for all $t\in\rr$ and 
	\[
	\inf_{y\in\rn}\norm{\vec u(t)-\vec\ff(\cdot-y)}_\x<\epsilon.
	\] 
	Otherwise, $\vec\ff$ is said to be orbitally unstable.
\end{definition}

Given $\vec\ff=(\ff,\psi)\in\ngg(\vec c)$, we define
$
d(\vec c)=S_{\vec c}(\vec \ff).
$
It follows from Remark \ref{remark-vect} that
\begin{equation}\label{d-est}
	\begin{split}
		d(\vec c)&=\ns(\ff)=m_1=J( \ff )-K(\ff)=\frac{p-1}{2}K(\ff)=\frac{p-1}{2}\left(\frac{2m(\vec c)}{p+1}\right)^{\frac{p+1}{p-1}}\\&=
		\frac{p-1}{n+2-p(n-2)}\left(\|\ff\|_\lt^2-\norm{\ltt\ff}_\lt^2\right)\\
		&=\frac{p-1}{2(p+1)}\left(\|\ff\|_\lt^2-\norm{\ltt\ff}_\lt^2+\|\nabla\ff\|_\lt^2\right)\\&=\frac{p-1}{n+2-p(n-2)}J(\ff)\\
		&\geq
		\frac{p-1}{2(p+1)}(1-|\vc|^2)\|\ff\|_{H^1(\rn)}^2\\
		&=\frac{p-1}{2(p+1)}(1-|\vc|^2)\left(J(\ff)
		+|\vc|^2\norm{\ltt\ff}_\lt^2\right).
	\end{split}
\end{equation}

So $d$ is independent of the choice of $\vec\ff\in\ngg(\vc)$. If $w=\L_{\vec c}f$, then $(\vec c\cdot \nabla)^2 f=-\de w$, so that by  the invariance of the Laplacian under orthogonal transformations, it is observed that $d(\vec c)=\bar{d}(|\vec c|)$, where
\[
\bar{d}(c)=\E(\vec\ff)-c\int_\rn\psi(\de^{-\frac12}\ff)_{x_1}\dd x
\]
with $\vec\ff=( \ff,\psi)$ and $|c|<1$. Therefore $d$ is radial.

\begin{lemma}\label{lemma-diff-dc}
	Let $\vc=\zeta\vs$ with $\vs\in B_1(0)\subset\rn$, and $d_\vs=d(\zeta\vs)$ for $\zeta\in [0,1)$. If $\vec\ff=(\ff,\psi)\in\ngg(\vec c)$, then $d_\vs$  is continuous and strictly decreasing on $[0,1)$. Moreover, it is differentiable  at almost every $\zeta\in [0,1)$.  We have at points of differentiability that 
	\[
	d_\vs'(\zeta)=-\zeta \norm{\L_\vs^{\frac12}\ff} _\lt^2
	\]
	for any $\ff\in G_1(\vc)$. Moreover, we have in general for any $\vc\in B_1(0)$ that
	\[
	\nabla_\vc d=-\F(\vec\ff).
	\]
\end{lemma}
\begin{proof}
	To clarify the dependence on the wave velocity $\vc$, we write $J(\ff;\vc)$, $\ns(\ff;\vc)$ and $\PP(\ff;\vc)$.
	
	We know for any $\ff\in G_1(\vc)$ that
	\[
	d_\vs(\zeta)=\ns(\ff;\vc)=\frac{p-1}{p+1}J(\ff;\vc) 
	=\frac{(p-1)(1-\zeta^2)}{2(p+1)}\norm{\L_\vs^{\frac12}\ff}_\lt^2.
	\]
	Thus,
	\[
	\frac{2(p+1)}{(p-1)(1-\zeta^2)}d_\vs(\zeta)
	\geq\al_+(\zeta)\geq\al_-(\zeta)\geq0,
	\]
	where\[
	\al_\pm(\zeta)=\pm\sup\left\{\pm\left\|\L_\vs^{\frac12}\ff\right\|_\lt^2,\;\ff\in G_1(\vc)\right\}.
	\] 
 Let $0\leq\zeta_1<\zeta_2<\zeta_0<1$ and $\ff_j\in G_1(\vc_j)$ for $j=1,2$,  where $\vc_j=\zeta_j\vs$. Then, it follows from $\PP(\ff_1;\vc_1)=0$ that
	\[
	\PP(\ff_1;\vc_2)=\PP(\ff_1,\vc_1)+(\zeta_1^2-\zeta_2^2)\left\|\L_\vs^{\frac12}\ff_1\right\|_\lt^2<0.
	\]
	Hence, we have
	\[
\begin{split}
		d(\vc_2)&\leq\ns(\ff_1;\vc_1)+\frac12\frac{(p-1)(\zeta_1^2-\zeta_2^2)}{p+1}\left\|\L_\vs^{\frac12}\ff\right\|_\lt^2\\&
	=d(\vc_1)+ \frac12\frac{(p-1)(\zeta_1^2-\zeta_2^2)}{p+1}\left\|\L_\vs^{\frac12}\ff\right\|_\lt^2\\&<d(\vc_1).
\end{split}
	\]
	This means that $d_\vs$ is strictly decreasing, and then differentiable almost everywhere in $[0,1)$.
	
	To show the second claim of the lemma, we first observe that
	\[
	\PP(s\ff_1;\vc_2)=2s^2J(\ff;\vc_2-2s^{p+1})J(\ff_1;\vc_1)
	\]
	for any $s\in(0,1)$. Then, $\PP(s;\vc_2)>0$ if $s^{p-1}<\frac{J(\ff_1;\vc_2)}{J(\ff_1;\vc_1)}$. Hence, there exists $\left(\frac{J(\ff_1;\vc_2)}{J(\ff_1;\vc_1)}\right)^{\frac{1}{p-1}}<s_1<1$ such that $\PP(s_1\ff_1;\vc_2)=0$ and
	\[
	d(\vc_2)\leq\ns(s_1\ff_1;\vs_1)
	-\frac{s_1^2(\zeta_2^2-\zeta_1^2)}{2}\left\|\L_\vs^{\frac12}\ff_1\right\|_\lt^2.
	\]
	Since $\ns(s\ff_2;\vc_1)$ attains its maximum at $s=1$, then $\ns(s_1\ff_1;\vc_1)\leq\ns(\ff_1;\vc_1)=d(\vc_1)$. Moreover, 
	\[
	d(\vc_2)\leq d(\vc_1)-\frac{s_1^2(\zeta_2^2-\zeta_1^2)}{2}\left\|\L_\vs^{\frac12}\ff_1\right\|_\lt^2.
	\]
	Using the fact $\left|\L_\vc^{\frac12}\right|\leq|\vc|$, we get $s_1\geq\left(1-\frac{\zeta_2^2-\zeta_1^2}{1-\zeta_0^2}\right)^{\frac{1}{p-1}}$ and
	that
	\[
	d(\vc_2)\leq d(\vc_1)
	-\frac{ (\zeta_2^2-\zeta_1^2)}{2}\left(1-\frac{\zeta_2^2-\zeta_1^2}{1-\zeta_0^2}\right)^{\frac{1}{p-1}}
	\left\|\L_\vs^{\frac12}\ff_1\right\|_\lt^2.
	\]
	This implies that
	\[
	d_\vs(\zeta_2)\leq d_\vs(\zeta_1)
	-\frac{ (\zeta_2^2-\zeta_1^2)}{2}\left(1-\frac{\zeta_2^2-\zeta_1^2}{1-\zeta_0^2}\right)^{\frac{1}{p-1}}\al_+(\zeta_1).
	\]
	In a similar fashion, we obtain that
	\[
	d_\vs(\zeta_1)\leq d_\vs(\zeta_2)
	+\frac{ (\zeta_2^2-\zeta_1^2)}{2}\left(1+\frac{\zeta_2^2-\zeta_1^2}{1-\zeta_0^2}\right)^{\frac{1}{p-1}}\al_-(\zeta_2).
	\]
	Overall, we are led to
	\[
	-\frac{\zeta_1+\zeta_2}{2}
	\left(1+\frac{\zeta_2^2-\zeta_1^2}{1-\zeta_0^2}\right)^{\frac{1}{p-1}}\al_-(\zeta_2)
	\leq\frac{d_\vs(\zeta_2)-d_\vs(\zeta_1)}{\zeta_2-\zeta_1}
	\leq 
	-\frac{\zeta_1+\zeta_2}{2}
	\left(1-\frac{\zeta_2^2-\zeta_1^2}{1-\zeta_0^2}\right)^{\frac{1}{p-1}}\al_+(\zeta_1)
	\]
	for $\zeta_1<\zeta_2<\zeta_0$.
	This reveals that $d_\vs$ is locally Lipschitz, and then continuous on $[0,1)$.  Now, by using the proof of Lemma 4.3 in \cite{leva-1998}, after natural modifications, we have
	$$
	\limsup _{\zeta_2\to\zeta_1}\al_-(\zeta_2)\leq \al_+(\zeta_1)
	$$ 
	and 
	$$
	\al_-(\zeta_2)\leq\liminf _{\zeta_1\to\zeta_2}\al_+(\zeta_1),
	$$ 
	so that $d_\vs'(\zeta^\pm)=-\zeta\al_\pm(\zeta)$. Consequently, $d_\vs$ is differentiable if and only if $\al_+(\zeta)=\al_-(\zeta)$. Hence, in the point of differentiability, we have $d_\vs'(\zeta)=-\zeta\|\L_\vs^{\frac12}\ff\|_\lt^2$ for any $\ff\in G_1(\zeta\vs)$.
\end{proof}

\begin{corollary}
	It holds for any $\vc\in B_1(0)$ that
	\[
	d(\vc)\geq(1-|\vc|^2)^{\frac{p+1}{p-1}}d(0).
	\]
\end{corollary}
\begin{proof}
	We assume again without loss of generality that $\vc=(c,0,\cdots,0)$.	It follows from \eqref{d-est} and Lemma \ref{lemma-diff-dc} that
	\[
	d(c)\geq(1-c^2)\left(d(c)-\frac{p-1}{2(p+1)}cd'(c)\right)
	\]
	which implies that
	\[
	d'(c)\geq -\frac{2c(p+1)}{(p-1)(1-c^2)}d(c).
	\]
	And the proof is complete.
\end{proof}
\begin{lemma}\label{lem-tayl}
	Let $\vc=\zeta\vs$ be as in Lemma \ref{lemma-diff-dc}, and suppose that $d_\vs''(|\vc|)>0$.
	Then there exists $\epsilon>0$ such that for all $\ff\in\ngg(\vc)$ and any $\vec u\in U_\epsilon(\ngg(\vc))\subset\x$ we have
	\[
	\E(\vec u)-\E(\vec\ff)+c(\vec u)\vs\cdot(\F(\vec u)-\F(\vec\ff))
	\geq \frac14d_\vs''(|\vc|)\paar{c_\vs(\vec u)-|\vc|}^2.
	\]
\end{lemma}
\begin{proof}
	Let $\vec\ff=(\ff,-\L_\vc^{\frac12}\ff)\in\ngg(\vec c)$. Since $d_\vs$ is strictly decreasing and $d(\vc)=\frac{p-1}{2}K(\ff)$,  we may associate to any $u\in U_\epsilon(G_1(\vc))\subset\ho$ the speed $c_\vs(\vec u)=d_\vs^{-1}(\frac{p-1}{2}K(u))\vs$. It is clear that $\PP(u;c_\vs(u))\geq0$ and $\ns(u; c_\vs(u))\geq d(c_\vs(u))$. Moreover, since we have from Lemma \ref{lemma-diff-dc} that $d_\vs'(|\vc|)=\vs\cdot\F(\vec\ff)$, then
	\[
	\begin{split}
		d_\vs(\zeta)&\geq d_\vs(|\vc|)+\vs\cdot\F(\vec\ff)(\zeta-|\vc|)+\frac14d_\vs''(|\vc|)(\zeta-|\vc|)^2\\
		&=\E(\vec\ff)
		+\zeta\vs\cdot\F(\vec\ff) +\frac14d_\vs''\paar{|\vc|)(\zeta-|\vc|}^2,\qquad|\zeta-|\vc||<\epsilon_0 
	\end{split}
	\]
	for some $\epsilon_0>0$. Now if $\vec u=(u,v)\in U_\epsilon(\ngg(\vc))\subset\x$, then $u\in U_\epsilon(G_1(\vc))\subset\ho$ and  
	\[
	d_\vs(\vs\cdot c_\vs(u) )\geq\E(\vec\ff)+c_\vs(u)\cdot\F(\vec\ff)
	+
	\frac14d_\vs''(|\vc|)\paar{\vs\cdot c_\vs(u)-|\vc|}^2.
	\]
	On the other hand, we have from the definition that
	\[
	d_\vs(\vs\cdot c_\vs(u))=d_\vs(c_\vs(u))\leq\ns(u;c_\vs(u))\leq\E(\vec u)+c_\vs(u)\cdot\F(\vec u).
	\]
	Combining the  last two inequalities completes the proof.
\end{proof}

\begin{theorem}\label{stability-t-1}
	Let $\vec  c\in B_1(0)$.  If $d_\vs''(|\vc|)>0$,  then $\ngg(\vec c)$  is stable, where $\vs=\vc/|\vc|$.
\end{theorem}
\begin{proof}
	The proof is similar to one of Theorem 5.4 in \cite{leva-1998} by using Lemma \ref{lem-tayl}. So we omit the details.
\end{proof}

Recall that $\ff_0\in H^1(\rn)$ is the unique positive radial solution  of \eqref{c=0}.

\begin{theorem}\label{lem3.2}
	Let $\{\vec c_k\}\subset B_1(0)$ be a sequence such that $\vec c_k\to\vec0$ as $k\to\infty$.
	If $\vec\ff_k\in\ngg(\vec c_k)$, then there exists a
	subsequence of $\vec c_k$ (still denoted by the same) and translations $y_k$ such that $\vec\ff_k(\cdot-y_k)\to(\ff_0,0)$ in $\x$.
\end{theorem}
\begin{proof}
	We have seen that if  $\{\vec\ff_k\}\subset \x$ satisfies $\frac{p-1}{p+1}J(\vec\ff_k)=\frac{p-1}{2}K(\vec\ff_k)\to d(\vec c)$, there exists some $\vec\ff\in\ngg(\vec c)$ and a sequence $\{y_k\}\subset\rn$ such that $\vec\ff_j(\cdot+y_k)$ converges, up to a subsequence, to $\vec\ff$ strongly in $\x$. In particular, this fact holds when $\vec c=\vec 0$. Now the continuity of $d$ shows that $m(\vec c)\to m(\vec 0)$. The proof of the theorem is thus concluded.
\end{proof}

\begin{lemma}\label{lem2.1}
	Let $\vec\ff\in\ngg(\vec c)$ with $\vc\in B_1(0)$. Assume that there exists $\delta>0$ such that $\la \nl \vec u,\vec u\ra\geq\delta\|\vec u\|_{\x}^2$ for any $\vec u\in \x$ satisfying   $\la\vec\ff,\vec u\ra=\la\vec u,\vec\ff_{x_j}\ra=0$, $1\leq j\leq n$. Then there  exist   $C>0$  and $\varepsilon>0$ such that
	
	\begin{equation}\label{e-est}
		\E(\vec u)-\E(\vec\ff)\geq C\inf_{y\in\rn}\|\vec u-\vec\ff(\cdot-y)\|_\x^2
	\end{equation}
	for any $\vec u\in U_\varepsilon(\vec \ff)$ satisfying $\vc\cdot\F(\vec\ff)=\vc\cdot\F(\vec u)$, where $U_\varepsilon(\vec\ff)$ is the  open $\varepsilon$-neighborhood of $\vec\ff$.
\end{lemma}
\begin{proof}
	Let $\vec u\in U_\varepsilon(\vec\ff)$ such that $\vc\cdot\F(\vec\ff)=\vc\cdot\F(\vec u)$. By the implicit function theorem, there is $y=y_{\vec u}\in\rn$ such that
	\[
	\|\vec u-\vec\ff(\cdot+y)\|_\x^2=\min_{z\in\rn}\|\vec u-\vec\ff(\cdot+z)\|_\x^2,
	\]
	provided $\varepsilon>0$ is small enough. The Taylor expansion implies that
	\[
	\svc(\vec u)=\svc(\vec\ff)+\la \svc'(\vec\ff),\vec w\ra+\frac12\la\nl(\vec\ff)\vec w,\vec w\ra+o(\|\vec w\|_\x^2),
	\] 
	where $\vec w=\vec u(\cdot-y)-\vec \ff$ and 	the linearized operator $\nl= \svc''(\ff)$ is defined by 
	\begin{equation}\label{second-der}
		\nl=  \left(\begin{array}{cc}
			\varUpsilon_\vc&- \L_\vc^{\frac12}\\
			\L_\vc^{\frac12}&I 
		\end{array}\right),
	\end{equation}
	where $\varUpsilon_\vc=	I-\de-p\ff^{p-1}$. It is obvious that $\varUpsilon_\vc$ is self-adjoint from $\ho$ to $H^{-1}(\rn)$. Differentiating equation $S'_\vc(\vec\ff)=0$  with respect to $x_j$, gives $\nl(\ff_{x_j})=0$ for $1\leq j\leq n$.
	Since $\vec\ff$ is the critical point of $\svc$ and $\vc\cdot\F(\vec u)=\vc\cdot\F(\vec\ff)$, we have 
	\[
	\E(\vec u)-\E(\vec\ff)=\frac12\la\nl(\vec\ff)\vec w,\vec w \ra+o(\|\vec w\|_\x^2).
	\]
	By decomposing $\vec w$ as $\vec w=a\vec\ff+\sum_{j=1}^n \ell_j\vec\ff_{x_j}+\vec v$ such that $a,\ell_j\in\rr$ and $\vec v\in\x$ satisfies  $\la\vec\ff,\vec v\ra=\la\vec v,\vec\ff_{x_j}\ra=0$, $1\leq j\leq n$, we obtain from the Taylor expansion that
	\[
	\vc\cdot\F(\vec\ff)=\vc\cdot\F(\vec\ff)+\la\vc\cdot\F'(\vec\ff),\vec w\ra+O(\|\vec w\|_\x^2)
	\]
	and
	\[
	\la\vc\cdot\F'(\vec\ff),\vec w\ra=a\|\ltt\ff\|_\lt^2\cong a\|\ff\|_\lt^2;
	\]
	so that $a=O(\|\vec w\|_\x^2)$. Moreover, we have $0=\la\vec w,\vec\ff_{x_j}\ra$ which implies that $|\ell_j|\|\vec\ff_{x_j}\|_\x\leq\|\vec v\|_\x$ and
	\[
	\|\vec w\|_\x\leq (n+2)\|\vec v\|_\x+O(\|\vec w\|_\x^2).
	\]
	Hence, we deduce that 
	\begin{equation}\label{est-3}
		\|\vec v\|_x^2\geq\frac{1}{n+2}\|\vec w\|_\x^2+O(\|\vec w\|_\x^3).
	\end{equation}
	Now, the fact $\nl(\vec \ff_{x_j})=0$ with $1\leq j\leq n$, implies that
	\begin{equation}\label{est-4}
		\la\nl(\vec v),\vec v\ra=\la\nl(\vec w),\vec w\ra+O(\|\vec w\|_\x^3).
	\end{equation}
	The assumption of lemma reveals that there exists $\delta>0$ such that  $\la \nl \vec v,\vec v\ra\geq\delta\|\vec v\|_{\x}^2$. This together with \eqref{est-3} and \eqref{est-4} implies that
	\[
	\E(\vec u)-\E(\vec\ff)\geq\frac{\delta}{2(n+2)}\|\vec w\|_x^2+o(\|\vec w\|_x^2).
	\]
	Since $\vec\ff\in U_\varepsilon(\vec\ff)$ and $\|\vec w\|_\x<\varepsilon$, then there is $\varepsilon>0$, depending on $\delta$, such that
	\[
	\E(\vec u)-\E(\vec\ff)\geq\frac{\delta}{2(n+2)}\|\vec u-\vec\ff(\cdot+y)\|_x^2+o(\|\vec w\|_x^2).
	\] And the proof is complete.
\end{proof}

The following lemma was proved in Lemma 2.4 in \cite{maris}.
\begin{lemma}\label{lemma3.3}
	There  exists $\delta>0$ such that 
	$\la \varUpsilon_0  u,  u\ra\geq \delta\|  u\|_{H^1(\rn)}^2$ for any $u\in H^1(\rn)$ satisfying   
	$\la \ff_0,  u\ra=\la  u, \partial_{x_j}\ff_0 \ra=0$, $1\leq j\leq n$, where $\ff_0$ is the unique positive ground state of \eqref{c=0}.
\end{lemma}
\begin{lemma}\label{prop2.1}
	Let $\vec\ff\in\ngg(\vec c)$ with $\vc\in B_1(0)\subset\rn$. Assume that there exists $\delta>0$ such that $\la \nl \vec u,\vec u\ra\geq\delta\|\vec u\|_{\x}^2$ for any $\vec u\in \x$ satisfying   $\la\vec\ff,\vec u\ra=\la\vec u,\vec\ff_{x_j}\ra=0$, $1\leq j\leq n$. Then $\vec\ff$ is stable in $\x$.
\end{lemma}
\begin{proof}
	The proof is followed by Lemma \ref{lem2.1} and the proof of Theorem 3.5 in \cite{gss}.
\end{proof}
Drawing inspiration from the concepts outlined in \cite{estst}, we demonstrate the following lemma.
\begin{lemma}\label{lem3.1}
	Let $\vec\ff\in\ngg(\vec c)$ with $\vc\in B_1(0)\subset\rn$. Then there is $c_0\in(0,1)$ such that if $\vc\in B_{c_0}(0)$, then there exists $\delta>0$ such that $\la\nl \vec u,\vec u\ra\geq \delta\|\vec u\|_\x^2$ for any $\vec u\in\x$ satisfying   $\la\vec\ff,\vec u\ra=\la\vec u,\vec\ff_{x_j}\ra=0$, $1\leq j\leq n$.
\end{lemma}
\begin{proof}
 We prove by contradiction.	Suppose that  Lemma \ref{lem3.1} is false. Then there exists $\{\vc_j\}\subset\rn$ and $\{\vec u_j=(u_{j},v_j)\}\subset\x$ such that $|\vc_j|<1$, $\vc_j\to0$, $\la\vec u_j,\vec\ff_j\ra=\la\vec u_j,\partial_{x_j}\vec\ff_j\ra=0$, $\|\vec u_j\|_\x=1$ and
	\begin{equation}\label{est-55}
 		\lim_{j\to\infty} \scal{\mathscr{L}_{\vc_j}\vec u_j,\vec u_j} \leq0,	 
	\end{equation} 
	where $\vec\ff_j=(\ff_j,-\ltt\ff_j)\in\ngg(\vc_j)$. Theorem \ref{lem3.2} then implies that there exists a subsequence of $\{\vc_j\}$, still denoted by the same letter, and $\{y_j\}\in\rn$ such that $\vec\ff_j\to\vec\ff_0=(\ff_0,0)$ strongly in $\x$. Hence,  there is a subsequence of $\x$-bounded sequence of $\{\vec u_j(\cdot+y_j)\}$, still denoted by the same letter, and $\vec u=(u,v)\in\x$ such that $\vec u_j(\cdot+y_j)\rightharpoonup\vec u$ in $\x$  and $|\vec u_j(\cdot+y_j)|^2\rightharpoonup|\vec u|^2$ in $L^{\frac{p+1}{2}}(\rn)$. Hence, we obtain that
	\[
	\lim_{j\to\infty}\int_\rn\ff_j^{p-1}(x+y_j)|u_j(x+y_j)|^2\dd x=\int_\rn\ff_0^{p-1}|u|^2\dd x.
	\]
	This implies that
	\[
	0\geq\liminf_{j\to\infty}\la\mathscr{L}_{\vc_j}\vec u_j,\vec u_j\ra
	=1-p\int_\rn\ff_0^{p-1}|u_0|^2\dd x.
	\]
	Moreover, we have
	$
	0\geq\liminf_{j\to\infty}\la\mathscr{L}_{\vc_j}\vec u_j,\vec u_j\ra
	\geq
	\la \varUpsilon_0  u,u\ra.
	$
	Finally we have from the assumption   $\la\vec u_j,\vec\ff_j\ra=\la\vec u_j,\partial_{x_j}\vec\ff_j\ra=0$ for $1\leq j\leq n$,  that
	$\la  u , \ff_0\ra=\la  u,\partial_{x_j} \ff_0\ra=0$. Therefore Lemma \ref{lemma3.3} shows that $  u=0$. This contradiction  completes the proof.
\end{proof}

Combining the results of Lemma \ref{prop2.1} and Lemma \ref{lem3.1}, we obtain the following stability result.
\begin{theorem}\label{stab-theo}
	Let $\vec\ff\in\ngg(\vec c)$ with $\vc\in B_1(0)\subset\rn$. Then there is $c_0\in(0,1)$ such that any $\vec\ff\in\ngg(\vc)$ with $\vc\in B_{c_0}(0)$ is orbitally stable in $\x$.
\end{theorem}

It is worth noting that the proof of stability in \cite{gss} is based on the spectra of the linearized operator $\nl$. Due to the presence of $\ltt$ in $\nl$, deriving the spectral properties of $\nl$ is not straightforward.

In \cite{ber-lions}, it was demonstrated that $\ff_0$ is the unique ground state of \eqref{c=0} and decays exponentially. Specifically, $\ff_0$ belongs to $m(0)$, and $(\ff_0,0)\in\ngg(0)$. Additionally, it was proved in \cite{wein} that the operator $\varUpsilon_0:=I-\Delta-p\ff^{p-1}_0$ has a unique, simple, negative eigenvalue $\lam_0$ with a corresponding positive, radially symmetric eigenfunction $\chi_0$ that decays exponentially. Moreover, the essential spectrum of $\varUpsilon_0$ is $[\sigma,+\infty)$ for some $\sigma>0$, and the null space of $\varUpsilon_0$ is spanned by ${\partial_{x_j}\ff_0,\;1\leq j\leq n}$.

The operator $\nl$ with $\vec c=\vec0$ is denoted by $\nlz$. It is evident that $\nlz={\rm diag}(\varUpsilon_0, I)$, so the spectral properties of $\varUpsilon_0$ are inherited by $\nlz$. The perturbation theory, as outlined in \cite{kato}, can be applied to prove Theorem \ref{stab-theo} by using the approach from \cite{gss}.

\begin{proposition}\label{cpt-embed-lem}
	There exists $c_0\in(0,1)$ such that for any $\vc\in B_{c_0}(0)\subset\rn$, Equation \eqref{criticalpoint} has a solution $\vec\ff\in \x_\vc:=H_\vc^1(\rn)\times L^2_\vc(\rn)$, and the map $\vc\mapsto\vec\ff$ is continuous from $B_{c_0}(0)$ to $\x_\vc$, where $L^2_\vc(\rn)$ is defined analogously to $H^1_\vc(\rn)$ by replacing $\ho$ with $\lt$. Furthermore, $\vec\ff$ converges to $(\ff_0,0)$ as $\vc\to0$, uniformly in $x\in\rn.$
\end{proposition}
\begin{proof}
	Assume, without loss of generality in \eqref{criticalpoint}, that $\vc=(c,0,\cdots,0)$. Define the map $\wp:\rr\times\x_c\to L^2_c(\rn)\times L^2_c(\rn)$ by
	$
	\wp(c,\vec\ff)=S'_c(\vec\ff)$. Then, $\frac{\partial\wp}{\partial\vec\ff}$, presented in \eqref{second-der}, is defined from $\rr\times\x_c$ to $B(\x_c,L^2_c(\rn)\times L^2_c(\rn))$. It is clear that $(\ff_0,0)\in\x_c$  and $\wp(0,(\ff_0,0))=0$. Recall that ${\rm null}(\nlz)$ is an $n$-dimensional subspace of $L^2(\rn)\times L^2(\rn)$ generated by ${(\partial_{x_j}\ff_0,0),\;1\leq j\leq n}$. Since $\partial_{x_1}\ff\notin H_c^1(\rn)$, then $\nlz:\x_c\to L^2_c(\rn)\times L^2_c(\rn))$ is invertible. Hence, the continuity of $\wp$ and $\frac{\partial\wp}{\partial\vec\ff}$ implies from the Implicit Function Theorem that there is $c_0>0$ and a continuous map $c\to \vec\ff$ from $(-c_0,c_0)$ to $\x_c$ such that $\wp(c,\vec\ff)=0$ for all $c\in(-c_0,c_0).$
\end{proof}

Note that $\L_\vc$ is a compact operator  {on} $\lt$, so it follows from \cite{kato} that $\sigma_{\rm ess}(\nl)=\sigma_{\rm ess}(\nlz)$. Since both $\varUpsilon_\vc$ and $\nl$ vary in $\vc$ continuously in the space of closed operators, the proof follows from standard theorems (see e.g. Theorems 2.14 and 2.17 in \cite{kato}) on the continuity of eigenvalues of linear operators with respect to perturbations. To obtain some information about the spectrum of $\varUpsilon_\vc$, we need to recall some  notations  and definitions.

Let $C(\lt)$ be the space of all closed operators on $\lt$. A metric $\hat{\delta}$ on this space may be defined as follows: for any $S,T\in C(\lt)$,
\[
\hat{\delta}(S,T)=\|P_S-P_T\|_{B(\lt\times\lt)},
\]
where $P_T$ and $P_S$ are   the orthogonal projections on their graphs, and
$\norm{\cdot}_{B(\lt\times\lt)}$ denotes  the operator norm on the space of bounded operators
on $ \lt\times\lt$.
\begin{proposition}\label{spectrum-pertu}
	There exists $c_0\in(0,1)$ such that $\varUpsilon_\vc$  with any $\vc\in B_{c_0}(0)$, defined on  $\ho$, possesses the following properties:
	\begin{enumerate}[(1)]
		\item  $\varUpsilon_\vc$ has simple negative eigenvalue $\lam_\vc$ with positive, radially symmetric eigenfunction;
		\item The null space of $\varUpsilon_\vc$ is spanned by $\{\partial_{x_j}\ff,\;1\leq j\leq n\}$, where $\vec\ff=(\ff,\psi)$ is the traveling wave of \eqref{traveling};
		\item The essential spectrum of $\varUpsilon_\vc$ is positive and bounded away zero;
		\item The essential spectrum of $\varUpsilon_\vc$ is in $[\sigma,+\infty)$ and $(\lam_\vc,\chi_\vc)\to(\lam_0,\chi_0)$ in $\rr\times\lt$ as $\vc\to0$.
	\end{enumerate}
	Moreover,  the above spectral properties are also valid for $\nl$   with any $\vc\in B_{c_0}(0)$.
\end{proposition}
\begin{proof}
	First, we demonstrate that $\hat{\delta}(\varUpsilon_\vc,\varUpsilon_0)\to0$ as $\vc\to0$. It is worth noting that $\varUpsilon_\vc-\varUpsilon_0\in B(\lt)$ due to continuity, and $\varUpsilon_\vc-\varUpsilon_0\to0$ in $B(\lt)$ as $\vc\to0$. By applying Theorems 2.14 and 2.17 from \cite{kato}, we obtain:
	\[
	\hat{\delta}(\varUpsilon_\vc-\varUpsilon_0)
	\leq 2\left(1+\|\varUpsilon_\vc\|_{B(\lt)}^2\right)\|\varUpsilon_\vc-\varUpsilon_0\|_{B(\lt)}\to0,\quad\text{as}\;\vc\to0,
	\]
	 owing to the uniformly bounded nature of $\|\varUpsilon_\vc\|_{B(\lt)}$ with respect to $\vc$. 
\end{proof}

\section{Orbital instability}\label{section-orb-ins}
In this section, we establish  conditions under which the   traveling wave   $\vec\ff$ of \eqref{hbouss} is orbitally unstable. We use the approach developed in \cite{ribeiro}  based on ones in \cite{gss}. Let $O_{\vec\ff}$ be the orbit of   $\vec\ff=(\ff,-\ltt\ff)\in\ngg(\vc)$, that is 
\[
O_{\vec\ff}=\{\vec\ff(\cdot-r),\;r\in\rn\}.
\]

 Our main instability result reads as follows. 
\begin{theorem}\label{main-inst-thm}
	Let $\vec\ff\in\ngg(\vc)$ and $|\vc|<1$. If there exists $\vec\Phi\in \x$ such that $\la\F'(\vec\ff),\vec\Phi\ra=0$ and $$\scal{\nl\vec\Phi,\vec\Phi}<0,$$ then $O_{\vec\ff}$ is unstable, where
	$
	\nl=S_\vc''(\vec\ff).
	$
\end{theorem}

To prove this theorem, we assume, without loss of generality, that $\vc=(c,0,\cdots,0)$.

\begin{lemma}\label{lem-1-ins}
	Let $|c|<1$ and $\vec\ff\in\ngg(\vc)$. Then there exists $\epsilon>0$ and a unique $C^1$-map $\Lambda:B_\epsilon(O_{\vec\ff})\to\rr$ such that $\Lambda(\vec\ff)=0$ and $\Lambda(\tau_r\vec \psi)=\Lambda(\vec\psi)-r$. Moreover, there holds that
	\begin{enumerate}[(i)]
		\item $\scal{\tau_{\Lambda(\vec\psi)}\vec \psi,\njj\vec\ff}=0$,
		\item $\Lambda'(\vec\psi)=-\frac{\tau_{-\Lambda(\vec\psi)}\njj\vec\ff}{\scal{\F'(\vec\psi),\tau_{-\Lambda(\vec\psi)}\vec\ff}}$ 
	\end{enumerate}
	for any $\vec\psi\in B_\epsilon(O_{\vec\ff})$ and $r\in\rr$, where $\tau_\cdot\vec\psi(x)=\vec\psi(x-\cdot)$ and $B_\epsilon(O_{\vec\ff})$ is the $\epsilon$-neighborhood of $O_{\vec\ff}$ in $\x$.
\end{lemma}  
\begin{proof}
	Define the map $\Lambda_1:\rr\times B_\epsilon(O_{\vec\ff})\to\rr$ by \[
	\Lambda_1(r,\vec\psi)=\scal{\tau_r\vec\psi,\njj\vec\ff}.
	\]
	Then the skew-adjointness of $\njj$ implies that $\Lambda_1(0,\vec\ff)=0$. Moreover, it is easy to see that $\frac{\partial\Lambda _1}{\partial r}(0,\vec\ff)>0$, so the implicit function theorem shows the existence of the unique map $\Lambda$ such that $\Lambda(\vec\ff)$ and $\Lambda(\tau_r\vec \psi)=\Lambda(\vec\psi)-r$. Parts (i) and (ii) can be derived straightforwardly.
\end{proof}

Now we define a vector field similar to one defined in \cite{bss} to study the instability.
\begin{lemma}\label{lem-2-ins}
	Let $\Lambda$ be as in Lemma \ref{lem-1-ins} and $\vec\Phi$ as in Theorem \ref{main-inst-thm}.	There exists a $C^1$-map $\fb:B_\epsilon(O_{\vec\ff})\to \x$ such that $\tau_\cdot\fb=\tau_\cdot\fb$, $\fb(\vec\ff)=\vec\Phi$ and $\scal{\fb(\vec\psi),\F'(\vec\psi)}=0$ for all $\vec\psi\in B_\epsilon(O_{\vec\ff})$. Moreover, $\fb$ satisfies
	\begin{equation}
		\njj\fb(\vec\psi)=\tau_{-\Lambda(\vec\psi)}\njj\vec\Phi+\scal{\F'(\vec\psi),\tau_{-\Lambda(\vec\psi)}\vec\Phi}\Lambda'(\vec\psi).
	\end{equation}
	Furthermore, there exists a unique  local $C^1$-solution $\fu(s,\vec\psi_0)$ of the initial value problem $\frac{\dd\vec u}{\dd s}=\fb (\vec u)$ with $\fu(0,\vec\psi_0)=\vec\psi_0\in B_\epsilon(O_{\vec\ff})$  such that $\frac{\dd \fu}{ds}(0,\vec\ff)=\vec\Phi$, $\F(\fu(s,\vec\psi_0))=\F(\vec\psi_0)$ and
	\[
	\fu(s,\tau_r\vec\psi_0)=\tau_r\fu(s,\vec\psi_0) 
	\]
	for all $r\in\rr$ and $\vec\psi_0\in B_\epsilon(O_{\vec\ff})$.
\end{lemma}
\begin{proof}
	It is enough to define
	\[
	\fb(\vec\psi)=\tau_{-\Lambda(\vec\psi)}\vec\Phi-
	\frac{\scal{\F'(\vec\psi),\tau_{-\Lambda(\vec\psi)}}}
	{\scal{\F'(\vec\psi),\tau_{-\Lambda(\vec\ff)}}}
	\tau_{-\Lambda(\vec\psi)}\vec\ff.
	\]
	The proof is deduced from the properties of $\Lambda$ in Lemma \ref{lem-1-ins} by direct calculations, so we omit the details.
\end{proof}

\begin{proof}[Proof of Theorem \ref{main-inst-thm}]
	The proof follows from the same lines as one in \cite{fele}, but we include a sketch here for completeness.
	
	First, we show that there are $\epsilon>0$ and $\rho>0$ such that for all $\vec\psi_0\in B_\epsilon(O_{\vec\ff})$, there exists $s=s_{\vec\psi_0}\in(-\rho,\rho)$ such that
	\begin{equation}\label{conv-es}
		S(\vec\ff)\leq S(\vec\psi_0)+\scal{S'(\vec\psi_0),\fb(\vec\psi_0)}s.
	\end{equation}
	Indeed, by applying the Taylor expansion theorem, we obtain  that
	\[
	S(\fu(s,\vec\psi_0))
	=S(\vec\psi_0)
	+\scal{S'(\vec\psi_0),\fb(\vec\psi_0)}s+
	\frac{s^2}2\Theta(s,\vec\psi_0) ,
	\]
	where
	\[
	\Theta(s,\vec\psi_0)
	=
	\scal{\nl(\fu(rs,\vec\psi_0))\fb(\fu(rs,\vec\psi_0)),\fb(\fu(rs,\vec\psi_0))}+\scal{S'(\fu(rs,\vec\psi_0)),\fb'(\fu(rs,\vec\psi_0))(\fu(rs,\vec\psi_0))}.
	\]
	From the definition of $\fb$, it is evident that $\Theta(s, \vec\ff) < 0$, ensuring that
	$
	S(\fu(s,\vec\psi_0))\leq S(\vec\psi_0)+\frac{s^2}2\Theta(s,\vec\psi_0)
	$
	for all $\vec\psi\in B_\epsilon(O_{\vec\ff})$. This   implies that 
	\[
	S(\vec\ff)\leq S(\fu(\varrho,\vec\ff))
	-
	\scal{S'(\fu(\varrho,\vec\ff)),\fb(\fu(\varrho,\vec\ff))}\varrho
	\]
	for all $|\varrho|<\rho$. Furthermore, it follows from the assumptions of Theorem \ref{main-inst-thm} and the continuity of $\Theta$ and $\fu$ that $S(\fu(\varrho,\vec\ff))<S(\vec\ff)$ if $|\varrho|<\rho$ for sufficiently small $\rho>0$. Additionally, we have
	\begin{equation}\label{zt-est}
		\scal{S'(\fu(\varrho,\vec\ff)),\fb(\fu(\varrho,\vec\ff))}<0
	\end{equation}
	for all $\varrho\in(0,\rho)$. The desired result follows if we show that for each $\vec\psi_0$, there is $s$ such that $\PP(\fu(s,\psi_0))=0$ because $\vec\ff$ minimizes $S$ subject to  $\PP(\ff)$. This is derived by defining $g(s,\psi)=\PP(\fu(s,\psi))$ and utilizing the facts that $f(0,\vec\Phi)=0$, $\frac{\partial g}{\partial s}(0,\vec\ff)=\scal{\PP'(\vec\ff),\vec\Phi}\neq0$ (see \cite[Lemma 4.5]{fele}) and the assumption $\scal{\nl\vec\Phi,\vec\Phi}<0$.
		Now define
	\[
	A(\vec\psi)=\scal{\tau_{-\Lambda(\vec\psi)}\njj\vec\Phi,\vec\psi}.
	\]
	It is clear from Lemmas \ref{lem-1-ins} and \ref{lem-2-ins} that  $A:B_\epsilon(O_{\vec\ff})\to\rr$  is a $C^1$-map and $A'=\njj\fb$. Suppose, by contradiction, there is $\{\varrho_j\}\subset(0,,\rho)$ such that $\varrho_j\to0$ as $j\to\infty$ and $T_j=\sup\{t;\;\vec u_j(t)\in B_\epsilon(O_{\vec\ff})\}$ is unbounded as $j\to\infty$, where $\vec \ff_j=\fu(\varrho_j,\vec\ff)$ and $\vec u_j$ is the unique solution of \eqref{hbouss} with initial data $\vec\ff_j$. It is deduced from the continuity of $\fu$ that $\vec \ff_j\to\vec\ff$ in $\x$. The invariance of $S$ under the flow map of \eqref{system} reveals that $S(\vec\ff_j) < S(\vec\ff)$, and then \eqref{zt-est} leads to $S(\vec u_j) < S(\vec\ff)$ for all $t \in [0, T_j]$.  Moreover, $\scal{S'(\vec u_j), \fb(\vec u_j)} < 0$ for all $t \in [0, T_j]$, implies that $$\fd = \sett{\vec u \in \x,\; S(\vec u) < S(\vec\ff),\; \scal{S'(\vec u), \fb(\vec u)} < 0}$$ is invariant under the flow of \eqref{system}. The previous analysis shows that for each $t \in [0, T_j]$, there exists $s = s(\vec u_j(t)) \in (-\rho, \rho)$ satisfying \eqref{conv-es}. The invariance $S$ and the fact $\vec u_j(t)\in\fd$ imply for all $t\in[0,T_j]$ that 
	\[
	-\scal{S'(\vec u_j),\fb(\vec u_j)}\geq\frac{S(\vec\ff)-S(\vec\ff_j)}{\rho}=\chi_j>0.
	\]
	Since $\vec u_j(t)\in B_\epsilon(O_{\vec\ff})$, then $A(\vec u_j(t))$ is bounded on $[0,T_j]$. By Lemma \ref{lem-2-ins},
	\[
	\begin{split}
		\frac{\dd}{\dd t}A(\vec u_j(t))&=
		\scal{A'(\vec u_j(t)),\vec u_j'(t)}\\&
		=\scal{\njj\fb(\vec u_j(t)),\njj\E'(\vec u_j(t))}\\&
		=\scal{\fb(\vec u_j(t)),S'(\vec u_j(t))}=-\chi_j.
	\end{split}
	\]
	Hence, $A(\vec u_j(t)) \to -\infty$ as $t \to \infty$, indicating that $T_j$ must be finite. This contradiction completes the proof.
\end{proof}

\begin{corollary}
	Assume that there is a $C^2$-map $c\mapsto\vec\ff\in\ngg(\vc)$. If $\vec\ff\in\ngg(\vc)$ such that $d''(|\vc|)<0$, then $O_{\vec \ff}$  is  unstable.
\end{corollary}
\begin{proof}
	Define \[
	\vec\Phi=\vec\ff-\frac{2d'(|\vc|)}{d''(|\vc|)}\fd(\vec\ff),
	\]
	where
	$D(\vec\ff)=\frac{\dd}{\dd|c|}\vec\ff$.
	It is straightforward to see from Lemma \ref{lemma-diff-dc} that $\scal{\F'(\vec\ff),\Phi}=0$.
	On the other hand, we note that
	\[
	\begin{split}
		&\scal{\nl \fd(\vec\ff),\vec\ff}=-2d'(|\vc|),\qquad
		\scal{\nl \fd(\vec\ff),\fd(\vec\ff)}=-d''(|\vc|).
	\end{split}
	\]
	Thus, we get from 	$\scal{\nl \vec\Phi,\vec\Phi}$ and the assumption of Theorem \ref{main-inst-thm} that
	\[
	\scal{\nl \vec\Phi,\vec\Phi}
	=	\scal{\nl \vec\ff,\vec\ff}+
	4\frac{(d'(|\vc|))^2}{d''(|\vc|)}<0.
	\]
\end{proof}

The following corollary now completes the previous results of the case $\vc=0$.
\begin{corollary}\label{instability-cor}
	Let $\vec\ff\in\ngg(\vc)$ with $|\vc|\in(0,1)$. Then $O_{\vec\ff}$  is  unstable  if 
	\begin{equation}
		p>\frac{2-n+\sqrt{(n-2)^2+\varepsilon_\vc(4+2n+\varepsilon_\vc)}}{\varepsilon_\vc},\qquad \varepsilon_\vc=|\vc|^{-2}-1.
	\end{equation}
\end{corollary}
\begin{proof}
	For $\vec\ff=(\ff,-\ltt\ff)\in\ngg(\vc)$, define $\vec\Phi=(\ff,\ltt\ff)$.
	It is clear that $\scal{\F'(\vec\ff),\vec\Phi}=0$. So, by Theorem \ref{main-inst-thm}, it is enough to find conditions under which $\scal{\nl \vec\Phi,\vec\Phi}<0$.
	Recall  from \eqref{second-der}  that
	\[
	\nl =\left(\begin{array}{cc}
		I-\Delta-p\ff^{p-1}	&-\ltt\\
		\ltt	& I
	\end{array}\right).
	\]
	It follows from the definition of $\nl$  for  any $\vec u\in\x$ that
	\[
	\scal {   \nl \vec u,\vec u}
	=2J(u)-\scal{ K''(\ff)u,u}+2q(\vec u).
	\]
	Then we obtain from \eqref{pohoz} after some calculations that
	\[
	\begin{split}
		\scal{\nl \vec\Phi,\vec\Phi}&
		=
		(1-p)\|\ff\|_{L^{p+1}(\rn)}+4\norm{\ltt\ff}_\lt^2	\\
		&=\frac{2(1-p^2)}{2n+(p+1)(2-n)}\left(\|\ff\|_\lt^2-\norm{\ltt\ff}_\lt^2\right)+4\norm{\ltt\ff}_\lt^2\\
		&=\frac{2(1-p^2)}{2n+(p+1)(2-n)}\|\ff\|_\lt^2+ \frac{2(p+1)(3-2n+p)+8n}{2n+(p+1)(2-n)}\norm{\ltt\ff}_\lt^2.
	\end{split}
	\]
	Using  $|\ltt\ff|\leq|\vc||\ff|$, we conclude that
	\[
	\begin{split}
		\scal{\nl \vec\Phi,\vec\Phi}
		&\leq \frac{2|\vc|^{-2}(1-p^2)+2(p+1)(3-2n+p)+8n}{(p+1)(2-n)+2n}
		\norm{\ltt\ff}_\lt^2<0 
	\end{split}
	\]
	provided
	\[
	p>\frac{2-n+\sqrt{(n-2)^2+\varepsilon_\vc(4+2n+\varepsilon_\vc)}}{\varepsilon_\vc}.
	\]
\end{proof}

\section{Strong instability}\label{section-str-ins}
In this section, we study the strong instability of the  traveling waves  of \eqref{hbouss}. More precisely, we show that there is a sequence of initial data approaching the   traveling wave   while the  associated solutions of \eqref{hbouss} blow up in finite time.

The main result of this section is articulated as follows.

\begin{theorem}\label{inst-theo}
	Assume that $\vc\in B_1(0)\subset\rn$. Suppose also that  $1<p<2^\ast-1$ and $|\vec c|^2<\frac{p-1}{p+3}$. 
	Then any $\vec \ff=(\ff,-\ltt\ff)\in \ngg(\vc)$ of \eqref{hbouss} is strongly unstable,  that is,  for any $\epsilon>0$, there exists $\vec u_0=(u_0,u_1)\in\x$ such that $\|\vec u_0-\vec\ff\|_\x<\epsilon$ and the solution $\vec u(t)$ of \eqref{hbouss} with
	initial value $\vec u_0$ blows up in finite time. 
\end{theorem}

\begin{figure}[ht]
	\begin{center}
		\scalebox{0.28}{\includegraphics{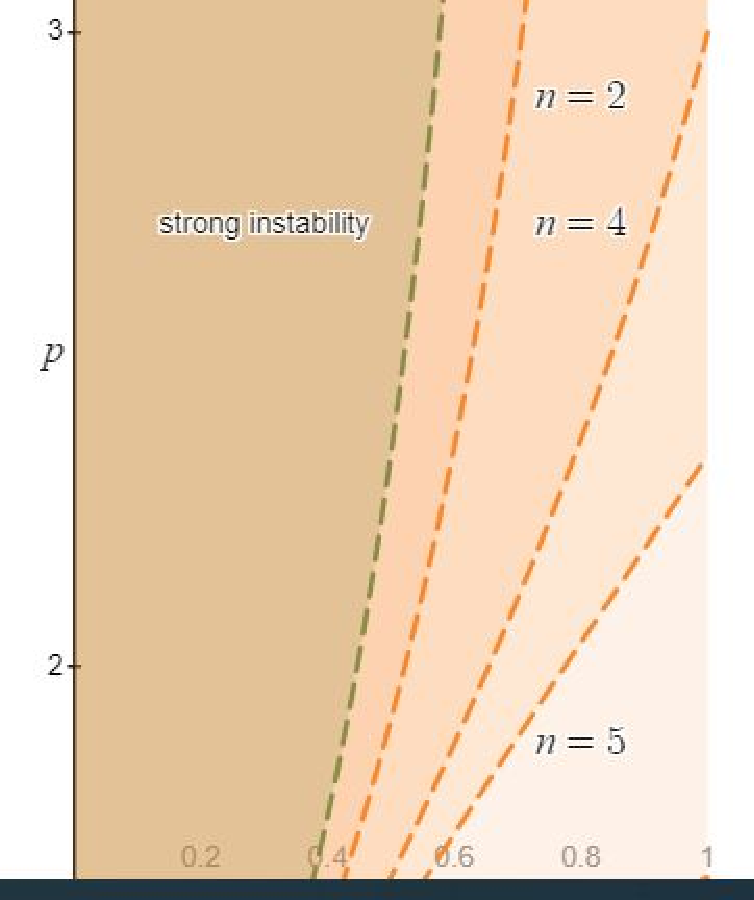}   	}
		\caption{The regions of orbital instability and strong instability for different $n=2,5,8$ and $|c|\in[0,1)$. } \label{nume-slt-2}
\end{center}\end{figure}

Following \cite{jeantanaka}, it is clear that the functional $\ns$ admits a mountain
pass geometry with $|\vec c|<1$, and also the mountain pass value gives the least energy level. Since $\ltt$ is homogeneous of the zeroth-order, we announce the following  lemma and refer  to \cite{jeantanaka} for the proof.
\begin{lemma}\label{mountain}
	There holds that $c_m=c_g$, where $c_g=\inf\{\ns(\ff),\;\ff\in\ho\setminus\{0\},\;\ns'(\ff)=0\}$ and 
	\[
	c_m=\inf_{\gamma\in\Gamma}\max_{t\in[0,1]}\ns(\gamma(t))>0 
	\] 
	with  
	$
	\Gamma=\{\gamma\in C([0,1],\ho),\;\gamma(0)=0,\;\ns(\gamma(1))<0\}.
	$
\end{lemma}

\begin{lemma}\label{equiv-min}
Let $\ff$ be in $G(\vec c)$ with $\vc\in B_1(0)\subset\rn$. Assume that $1<p<2^\ast-1$ and $|\vec c|^2<\frac{p-1}{p+3}$. 
	Then \[
	\ns(\ff)=\inf
	\left\{\ns(v),\;v\in\ho\setminus\{0\},\;\nk(v)=0\right\},
	\]
	where
	\[
	\nk(u)=\frac{2\al-\beta n}2\left(\|u\|_{\lt}^2-\|\ltt u\|_\lt^2\right)+\frac{2\al-\beta (n-2)}2\|\nabla u\|_\lt^2-(\al(p+1)-\beta n)K(u)
	\]
	such that $\alpha(p-1)\geq2\beta$ and
	\[
	\begin{cases}
		2\al=\beta n,&\beta>0,\\
		2\al>\beta n,&\beta\geq0.
	\end{cases}
	\]

\end{lemma}
\begin{proof}
	The proof comes from Theorems 1 and 2 in \cite{lecoz}, and we omit the details.
\end{proof}

\begin{lemma}\label{invar-lem-1}
	The set
	$
	\fk^-_\vc =\{\vec u=(u,v)\in\x,\;S(\vec u)<d(\vec c),\;\nk(u)<0\}
	$
	is invariant under the flow  of \eqref{hbouss-ivp}.
\end{lemma}
\begin{proof}
	Let $\vec u_0=(u_0,u_1)\in \fk^-_\vc$, and $\vec u(t)=(u(t),v(t))\in C([0,T),\x)$ be the unique solution of \eqref{hbouss-ivp} corresponding to $\vec u_0$.
	First, we note from \eqref{vect-S} that 
	
	\begin{equation}\label{Ss}
		\ns(\ff)=\svc(\vec\ff)-\frac12\|\L_\vc^{\frac12}\ff+\psi\|_\lt^2\leq \svc(\vec \ff).
	\end{equation}
	
	The invariants $\E$ and $\F$ show that $\svc(\vec u(t))=\svc(\vec u_0)<d(\vec c)$ for any $t\in[0,T)$. Now suppose that there exists $t_0\in(0,T)$ such that $\nk(u(t_0) )=0$ and $\nk(u(t))<0$ for any $t\in[0,t_0)$. Hence, it  is easy to see that 
	\[
	\tilde{\ns}(u(t))\geq d(\vec c)>0 
	\] 
	for any $[0,t_0)$, where $\tilde{\ns}(u)=\ns(u)-\frac 1{\al(p+1)-\beta n}\nk(u)$.
	Consequently, $u(t_0)\neq0$. Hence, the fact  $\nk(u(t_0))=0$, inequality \eqref{Ss} and Lemma \ref{equiv-min} imply that $d(\vec c)\leq \ns(u(t_0))\leq \svc(\vec u(t))<d(\vec c)$. This contradiction shows that $\vec u(t)\in \fk^-_\vc$ for any $t\in[0,T)$.
\end{proof}

\begin{remark}
	From the definition of $\fk^-_\vc$, it is evident that $\lam\vec\ff= \lam(\ff,-\ltt\ff) \in \fk^-_\vc$ for any $\lam>1$. Furthermore, by utilizing Lemmas \ref{mountain} and \ref{equiv-min}, the definition of $G(\vec c)$, and Theorem \ref{theorem-variational-2}, it can be observed that $d(\vec c)=\ns(\ff)$.
\end{remark}

\begin{proposition}\label{global-res}
	The set
	$
	\fk^+_\vc =\sett{\vec u=(u,v)\in\x,\;S(\vec u)<d(\vec c),\;\nk(u)>0}
	$
	is invariant under the flow  of \eqref{hbouss-ivp}. Moreover, if $\vec u_0=(u_0,u_1)\in \fk^+_\vc$, then the local solution $\vec u(t)\in C([0,T),\x)$ of \eqref{hbouss-ivp}, corresponding to the initial data $\vec u_0$,   exists globally in time.
\end{proposition} 
\begin{proof}
	Analogous to the proof of Lemma \ref{invar-lem-1}, the invariance of $\fk^+_\vc$ can be deduced. To demonstrate the global extension of the local solution $\vec u(t)=(u(t),v(t))\in C([0,T),\x)$, it suffices to establish the a priori estimate $\|\vec u(t)\|_\x\leq C_T$ for all $t\in[0,T)$. The first part of this  proposition  asserts that $\vec u(t)\in \fk^+_\vc$, meaning $S(\vec u(t))<d(\vc)$ and $\nk(u(t))>0$. Consequently, it follows from these inequalities that 
	\[
	\begin{split}
		&\|u(t)\|_\lt^2+\|\nabla u(t)\|_\lt^2+\|v(t)\|_\lt^2 
		+2\vc\cdot\F(\vec u(t))-2d(\vc)\\
		&\quad<
		\frac{1}{\al(p+1)-\beta n}\left(( 2\al-\beta n)  \left(\|u\|_{\lt}^2-\|\ltt u\|_\lt^2\right)+ (2\al-\beta (n-2)) \|\nabla u\|_\lt^2 \right).
	\end{split}
	\]
	This immediately implies that
	\[
	\begin{split}
		\frac{\al(p-1)}{\al(p+1)-\beta n}\|u(t)\|_\lt^2&+\frac{\al(p-1)-2\beta}{\al(p+1)-\beta n}
		+(2\al-\beta n) \|\ltt u(t)\|_\lt^2+\|v(t)\|_\lt^2\\& 
		<d(\vc)-2\vc\cdot\F(\vec u_0) 
	\end{split}
	\]
	for all $t\in[0,T)$. We consequently conclude  that $\|\vec u(t)\|_\x^2\lesssim d(\vc)-2\vc\cdot\F(\vec u_0)$ for all  $t\in[0,T)$.
\end{proof}

\begin{remark}\label{invariant-lemma-2}
	Lemma \ref{invar-lem-1} and Proposition 	\ref{global-res} still hold if one uses \eqref{vect-S} and  substitutes  invariant sets $\fk_\vc^\pm$  with
	$$
	\fk^\pm_\vc =\sett{\vec u=(u,v)\in\x,\;S(\vec u)<d(\vec c),\;\mp\nk(\vec u)<0}
	,$$
	where
	\[
	\begin{split}
		\nk(u,v)&=\frac{2\al-\beta n}2\left(\|u\|_{\lt}^2-\|\ltt u\|_\lt^2+\|\ltt u+ v\|_\lt^2\right)\\
		&\qquad+\frac{2\al-\beta (n-2)}2\|\nabla u\|_\lt^2-(\al(p+1)-\beta n)K(u).
	\end{split}
	\]
	Moreover, it holds that $(u_0,v_0)\in \fk^-_\vc$. Hence, the associated solution $\vec u$ of \eqref{system} satisfies $$\nk(\vec u)<-((p+1)\alpha-n\beta)(d(\vc)-S(u_0,v_0)).$$ 
\end{remark}

Let $\vec\ff=(\ff,-\ltt\ff)\in\ngg(\vec c)$. By an argument, similar to one of  Lemma 4.4 in \cite{liu95}, we readily observe that  the set $D=\{u\in\ho,\;(-\de)^{-\frac12} u\in\lt\}$ is dense in $\ho$. Whence,    there is $\{\vec\ff_j=(\ff_j,-\ltt\ff_j)\}_j$ such that $\ff_j\in D$ and $\ff_j\to\ff$ in $\ho$ as $j\to\infty$. Set $\vec u_0=\lam\vec\ff_j$ for $\lam>1$. Then, we have
\[
\svc(\vec u_0)=\E(\lam\vec\ff)+\vc\cdot\F(\lam\vec\ff)+O(j^{-1}).
\]
Since $\lam\vec\ff\in \mathfrak{K}^-_\vc$, it is found that $\svc(\vec u_0)<d(\vec c)$ by choosing $j\geq j_0$ sufficiently large. Similarly, we have $\nk(\vec u_0)<\nk(\vec\ff)=0$ and
$-\vc\cdot\F(\vec u_0)> \|\ltt\ff\|_\lt^2$  if $j\geq j_0$ is large enough. This implies that $\vec u_0\in \fk^-_\vc$ for $j \geq j_0$. Now, for $\lam>1$, let $\vec u(t)=(u(t),v(t))$, with $t\in[0,T)$, be the solution of \eqref{hbouss} with the initial data $\vec u_0=\lam\vec\ff_{j_0}$.
Denote
\[
I(t)=\|(-\de)^{-\frac12}u\|_\lt^2,\qquad t\in[0,T).
\]
Since $\vec u_0\in D$, then $I(t)$ is well-defined.
Using the same lines of Lemma 2.10 of \cite{lot}, it is easy to see that
\[
I'(t)=2\left\la(-\de)^{-\frac12}u,v\right\ra
\]
and
\[
I''(t)=2\|v\|_\lt^2-2\|u\|_{\ho}^2+2\|u\|_{L^{p+1}(\rn)}^{p+1}.
\]
\begin{lemma}\label{virial-est}
	For any $\lam>1$, there exists $C_\lam>0$ such that\[
	I''(t)\geq(p+3)  \|v+\ltt  u\|_\lt^2+C_\lam
	\]
	for all $t\in[0,T)$.
\end{lemma}
\begin{proof}
	First, we note from  
	\[
	2\|u\|_{L^{p+1}(\rn)}^{p+1}=-2(p+1)\E(\vec u)+(p+1)\|u\|_{H^1(\rn)}^2
	+(p+1)\|v\|_\lt^2 
	\]
	that
	\[
	\begin{split}
		I''(t)&=(p-1)\|u\|_{H^1(\rn)}^2-2(p+1)\E(\vec u)+(p+3)\|v\|_\lt^2\\
		&=(p+3)\|v+\ltt  u\|_\lt^2 -(p+3)\|\ltt  u\|_\lt^2
		+(p-1)\|u\|_{H^1(\rn)}^2-2(p+1) \svc(\vec u)-4\vc\cdot \F(\vec u)\\
		&=
		(p+3)\|v+\ltt u\|_\lt^2 -(p+3)\|\ltt  u\|_\lt^2
		+(p-1)\|u\|_{\ho}^2-2(p+1) \svc(\vec u_0)-4\vc\cdot \F(\vec u_0).
	\end{split}
	\]
	On the other hand, we have
	\[
	\begin{split}
		-(p+3)\|\ltt  u\|_\lt^2
		+(p-1)\|u\|_{\ho}^2
		&\geq
		-(p+3)|\vec c|^2\|   u\|_\lt^2
		+(p-1)\|u\|_{\ho}^2\\
		&=((p-1)-|\vec c|^2(p+3))\|u\|_\lt^2
		+(p-1)\|\nabla u\|_\lt^2\\&
		\geq	((p-1)-|\vec c|^2(p+3))(\|u\|_\lt^2-\|\ltt u\|_\lt^2)
		\\&\qquad	+(p-1)\|\nabla u\|_\lt^2\\ 
		&= \frac{2(p-1)(\al(p+1)-n\beta)}{\al(p-1)-2\beta}\tilde{\ns}(u).
	\end{split}
	\]
	Hence, we get
	\[
	I''(t)\geq 
	(p+3)\|v+\ltt u\|_\lt^2 + \frac{2(p-1)(\al(p+1)-n\beta)}{\al(p-1)-2\beta}\tilde{\ns}(u)-2(p+1) \svc(\vec u_0)-4\vc\cdot \F(\vec u_0).
	\]
	
	For any $\lam>1$, it is easy to see that $\svc(\vec u_0)=\ns(\lam\ff)<d(\vec c) $ and $-\vc\cdot\F(\vec u_0)>\|\ltt \ff\|_\lt^2$.
	By \eqref{pohoz}, it follows from Lemma \ref{invar-lem-1}, after some straightforward computation, that
	\[
	I''(t)\geq\|v+\ltt  u\|_\lt^2+\frac{2(p-1)(\al(p+1)-n\beta)}{\al(p-1)-2\beta}\left(\tilde{\ns}(u)-d(\vec c)\right)+\nc(u) ,
	\]
	where
	\begin{align*}
		\nc(u)&=2\left(\frac{ (p-1)(\al(p+1)-n\beta)}{\al(p-1)-2\beta}-p-1\right)d(\vec c)+4\|\ltt \ff\|_\lt^2\\
		&\geq 4\|\ltt \ff\|_\lt^2>0,
	\end{align*}
	provided $p<2^\ast-1$.
\end{proof}

\begin{proof}[Proof of Theorem \ref{inst-theo}]
	The proof follows from Lemma \ref{virial-est} and concavity arguments due to Levine \cite{levine} as in Payne and Sattinger \cite{payne-s}. Indeed, if $\lambda > 1$, then $\lambda\vec\ff$ converges to $\vec\ff$ in $\x$. By the density of $D$ in $\ho$, we can assume that the unique solutions $\vec u_\lambda\in C([0,T_\lambda),\x)$ of \eqref{hbouss}, corresponding to the initial data $\lambda\vec\ff$, satisfy the following estimate:
	\[
	I''(t)I(t)-\frac{p-1}{4}(I'(t))^2\geq0,
	\]
	presuming that $T_\lambda = +\infty$. However, this easily implies that there exists some $t_0$ where $I(t_0) \leq 0$. This contradiction completes the proof.
\end{proof}
 Combining \eqref{vector-gn} and \eqref{gn-equiv} and following the proof of Theorem \ref{local}, we derive a new blow-up condition depending on $d(\vec c)$. 
\begin{corollary}\label{blow-up-cor}
	Let $\vec c\in B_1(0)\in\rn$, $(u_0,u_1)\in H^1(\rn)\times \dot{H}^{-1}(\rn)$. 
	If  \[\|u_0\|_\ho>C_{\vec c}^{-\frac{p+1}{p-1}},\] then the local solution $u(t)\in\ho$ of \eqref{hbouss-ivp} blows up in finite time.   
\end{corollary}

\section*{Acknowledgments}
The author would like to thank the referees for their careful reading and valuable
comments improving presentation of the results. The author is supported by Nazarbayev University under Faculty Development Competitive Research Grants Program  for 2023-2025 (grant number 20122022FD4121).

\section*{Conflict of interest} The author  has no conflicts of interest to declare.

\section*{Data availability statement} 
Data sharing is not applicable to this article as no new data were created or analyzed in this study.


\bibliographystyle{acm}
\bibliography{main}	

\begin{thebibliography}{10}

\bibitem{ber-lions}
{\sc Berestycki, H., and Lions, P.-L.}
\newblock Nonlinear scalar field equations. {I}: {Existence} of a ground state.
\newblock {\em Arch. Ration. Mech. Anal. 82\/} (1983), 313--345.

\bibitem{bonali}
{\sc Bona, J.~L., and Li, Y.~A.}
\newblock Decay and analyticity of solitary waves.
\newblock {\em J. Math. Pures Appl. (9) 76}, 5 (1997), 377--430.

\bibitem{bonasachs}
{\sc Bona, J.~L., and Sachs, R.~L.}
\newblock Global existence of smooth solutions and stability of solitary waves
  for a generalized {Boussinesq} equation.
\newblock {\em Commun. Math. Phys. 118}, 1 (1988), 15--29.

\bibitem{bss}
{\sc Bona, J.~L., Souganidis, P.~E., and Strauss, W.~A.}
\newblock Stability and instability of solitary waves of {Korteweg}-de {Vries}
  type.
\newblock {\em Proc. R. Soc. Lond., Ser. A 411\/} (1987), 395--412.

\bibitem{bouss}
{\sc Boussinesq, J.}
\newblock Theory of waves and eddies propagating along a horizontal rectangular
  channel, communicating to the liquid in the channel approximately the same
  velocities from surface to bottom.
\newblock {\em Liouville J. (2) 17\/} (1872), 55--109.

\bibitem{blss}
{\sc Bugiera, L., Lenzmann, E., Schikorra, A., and Sok, J.}
\newblock On symmetry of traveling solitary waves for dispersion generalized
  {NLS}.
\newblock {\em Nonlinearity 33}, 6 (2020), 2797--2819.

\bibitem{haj-chang-song}
{\sc Chang, X., Hajaiej, H., Ma, Z., and Song, L.}
\newblock Existence and instability of standing waves for the biharmonic
  nonlinear {Schr{\"o}dinger} equation with combined nonlinearities.
\newblock {\em arXiv:2305.00327\/} (2023).

\bibitem{chalen}
{\sc Charlier, C., and Lenells, J.}
\newblock The ``{Good}'' {Boussinesq} equation: a {Riemann}-{Hilbert} approach.
\newblock {\em Indiana Univ. Math. J. 71}, 4 (2022), 1505--1562.

\bibitem{cgs}
{\sc Chen, J., Guo, B., and Shao, J.}
\newblock Well-posedness and scattering for the generalized {Boussinesq}
  equation.
\newblock {\em SIAM J. Math. Anal. 55}, 1 (2023), 133--161.

\bibitem{haj-cho-ozawa}
{\sc Cho, Y., Hajaiej, H., Hwang, G., and Ozawa, T.}
\newblock On the orbital stability of fractional {Schr{\"o}dinger} equations.
\newblock {\em Commun. Pure Appl. Anal. 13}, 3 (2014), 1267--1282.

\bibitem{ChristouChristov}
{\sc Christou, M.~A., and Christov, C.~I.}
\newblock Fourier-{Galerkin} method for 2{D} solitons of {Boussinesq} equation.
\newblock {\em Math. Comput. Simul. 74}, 2-3 (2007), 82--92.

\bibitem{ChrisC}
{\sc Christov, C., and Choudhury, J.}
\newblock Perturbation solution for the 2{D} {Boussinesq} equation.
\newblock {\em Mechanics Research Communications 38}, 3 (2011), 274--281.

\bibitem{ChristovTodorovChristou}
{\sc Christov, C., Todorov, M., and Christou, M.}
\newblock Perturbation solution for the 2{D} shallow-water waves.
\newblock In {\em AIP Conference Proceedings\/} (2011), vol.~1404, American
  Institute of Physics, pp.~49--56.

\bibitem{esfahani}
{\sc Esfahani, A., and Farah, L.~G.}
\newblock Local well-posedness for the sixth-order {Boussinesq} equation.
\newblock {\em J. Math. Anal. Appl. 385}, 1 (2012), 230--242.

\bibitem{estst}
{\sc Esteban, M.~J., and Strauss, W.~A.}
\newblock Nonlinear bound states outside an insulated sphere.
\newblock {\em Commun. Partial Differ. Equations 19}, 1-2 (1994), 177--197.

\bibitem{fele}
{\sc Feng, W., and Levandosky, S.}
\newblock Stability of solitary waves of a nonlinear beam equation.
\newblock {\em J. Differ. Equations 269}, 11 (2020), 10037--10072.

\bibitem{ribeiro}
{\sc Gon{\c{c}}alves~Ribeiro, J.~M.}
\newblock Instability of symmetric stationary states for some nonlinear
  {Schr{\"o}dinger} equations with an external magnetic field.
\newblock {\em Ann. Inst. Henri Poincar{\'e}, Phys. Th{\'e}or. 54}, 4 (1991),
  403--433.

\bibitem{gravej2004}
{\sc Gravejat, P.}
\newblock Decay for travelling waves in the {Gross}-{Pitaevskii} equation.
\newblock {\em Ann. Inst. Henri Poincar{\'e}, Anal. Non Lin{\'e}aire 21}, 5
  (2004), 591--637.

\bibitem{gravej2005}
{\sc Gravejat, P.}
\newblock Asymptotics for the travelling waves in the {Gross}-{Pitaevskii}
  equation.
\newblock {\em Asymptotic Anal. 45}, 3-4 (2005), 227--299.

\bibitem{gss}
{\sc Grillakis, M., Shatah, J., and Strauss, W.}
\newblock Stability theory of solitary waves in the presence of symmetry. {I}.
\newblock {\em J. Funct. Anal. 74\/} (1987), 160--197.

\bibitem{haj-2012}
{\sc Hajaiej, H.}
\newblock Orbital stability of standing waves of some {{\(\ell \)}}-coupled
  nonlinear {Schr{\"o}dinger} equations.
\newblock {\em Commun. Contemp. Math. 14}, 6 (2012), 1250039, 10.

\bibitem{haj-2013}
{\sc Hajaiej, H.}
\newblock On the optimality of the assumptions used to prove the existence and
  symmetry of minimizers of some fractional constrained variational problems.
\newblock {\em Ann. Henri Poincar{\'e} 14}, 5 (2013), 1425--1433.

\bibitem{hpl}
{\sc Hajaiej, H., Pacherie, E., and Song, L.}
\newblock On the number of normalized ground state solutions for a class of
  elliptic equations with general nonlinearities and potentials.
\newblock {\em arXiv:2308.14599\/} (2023).

\bibitem{hS}
{\sc Hajaiej, H., and Song, L.}
\newblock Strict monotonicity of the global branch of solutions in the
  {{\(L^2\)}} norm and uniqueness of the normalized ground states for various
  classes of {PDE}s: Two general methods with some examples.
\newblock {\em arXiv:2302.09681\/} (2023).

\bibitem{haj-st-2004}
{\sc Hajaiej, H., and Stuart, C.~A.}
\newblock On the variational approach to the stability of standing waves for
  the nonlinear {Schr{\"o}dinger} equation.
\newblock {\em Adv. Nonlinear Stud. 4}, 4 (2004), 469--501.

\bibitem{lecoz}
{\sc Jeanjean, L., and Le~Coz, S.}
\newblock Instability for standing waves of nonlinear {Klein}-{Gordon}
  equations via mountain-pass arguments.
\newblock {\em Trans. Am. Math. Soc. 361}, 10 (2009), 5401--5416.

\bibitem{jeantanaka}
{\sc Jeanjean, L., and Tanaka, K.}
\newblock A remark on least energy solutions in {{\(\mathbb R^N\)}}.
\newblock {\em Proc. Am. Math. Soc. 131}, 8 (2003), 2399--2408.

\bibitem{kato}
{\sc Kato, T.}
\newblock {\em Perturbation theory for linear operators. 2nd ed}, vol.~132 of
  {\em Grundlehren Math. Wiss.}
\newblock Springer, Cham, 1976.

\bibitem{leva-1998}
{\sc Levandovsky, S.}
\newblock Stability and instability of fourth-order solitary waves.
\newblock {\em J. Dyn. Differ. Equations 10}, 1 (1998), 151--188.

\bibitem{levine}
{\sc Levine, H.~A.}
\newblock Instability and nonexistence of global solutions to nonlinear wave
  equations of the form {{\(Pu_{tt}=-A u+ {\mathfrak F} (u)\)}}.
\newblock {\em Trans. Am. Math. Soc. 192\/} (1974), 1--21.

\bibitem{lowx}
{\sc Li, B., Ohta, M., Wu, Y., and Xue, J.}
\newblock Instability of the solitary waves for the generalized {Boussinesq}
  equations.
\newblock {\em SIAM J. Math. Anal. 52}, 4 (2020), 3192--3221.

\bibitem{linares}
{\sc Linares, F.}
\newblock Global existence of small solutions for a generalized {Boussinesq}
  equation.
\newblock {\em J. Differ. Equations 106}, 2 (1993), 257--293.

\bibitem{lions}
{\sc Lions, P.-L.}
\newblock The concentration-compactness principle in the calculus of
  variations. {The} limit case. {I}.
\newblock {\em Rev. Mat. Iberoam. 1}, 1 (1985), 145--201.

\bibitem{liu93}
{\sc Liu, Y.}
\newblock Instability of solitary waves for generalized {Boussinesq} equations.
\newblock {\em J. Dyn. Differ. Equations 5}, 3 (1993), 537--558.

\bibitem{liu95}
{\sc Liu, Y.}
\newblock Instability and blow-up of solutions to a generalized {Boussinesq}
  equation.
\newblock {\em SIAM J. Math. Anal. 26}, 6 (1995), 1527--1546.

\bibitem{lot}
{\sc Liu, Y., Ohta, M., and Todorova, G.}
\newblock Strong instability of solitary waves for nonlinear {Klein}-{Gordon}
  equations and generalized {Boussinesq} equations.
\newblock {\em Ann. Inst. Henri Poincar{\'e}, Anal. Non Lin{\'e}aire 24}, 4
  (2007), 539--548.

\bibitem{Lizorkin}
{\sc Lizorkin, P.~I.}
\newblock On multipliers of {Fourier} integrals in the spaces
  {{\(L_{p,\theta}\)}}.
\newblock {\em Proc. Steklov Inst. Math. 89\/} (1967), 269--290.

\bibitem{maris}
{\sc Mari{\c{s}}, M.}
\newblock Existence of nonstationary bubbles in higher dimensions.
\newblock {\em J. Math. Pures Appl. (9) 81}, 12 (2002), 1207--1239.

\bibitem{payne-s}
{\sc Payne, L.~E., and Sattinger, D.~H.}
\newblock Saddle points and instability of nonlinear hyperbolic equations.
\newblock {\em Isr. J. Math. 22\/} (1976), 273--303.

\bibitem{Pel-Step}
{\sc Pelinovsky, D., and Stepanyants, Y.~A.}
\newblock Solitary wave instability in the positive-dispersion media described
  by the two-dimensional {Boussinesq} equations.
\newblock {\em Zh. Eksp. Teor. Fiz 106\/} (1994), 192--206.

\bibitem{Skrzypczak}
{\sc Skrzypczak, L.}
\newblock Function spaces in presence of symmetries: compactness of embeddings,
  regularity and decay of functions.
\newblock In {\em Function Spaces, Differential Operators and Nonlinear
  Analysis: The Hans Triebel Anniversary Volume}. Springer, 2003, pp.~453--466.

\bibitem{wein-1983}
{\sc Weinstein, M.~I.}
\newblock Nonlinear {Schr{\"o}dinger} equations and sharp interpolation
  estimates.
\newblock {\em Commun. Math. Phys. 87\/} (1983), 567--576.

\bibitem{wein}
{\sc Weinstein, M.~I.}
\newblock Lyapunov stability of ground states of nonlinear dispersive evolution
  equations.
\newblock {\em Commun. Pure Appl. Math. 39\/} (1986), 51--67.

\end{thebibliography}

\end{document}